\documentclass[a4paper]{article}

\usepackage{etoolbox}
\appto\normalsize{\belowdisplayshortskip=\belowdisplayskip}
\appto\small{\belowdisplayshortskip=\belowdisplayskip}
\appto\footnotesize{\belowdisplayshortskip=\belowdisplayskip}
\appto\normalsize{\abovedisplayshortskip=\abovedisplayskip}
\appto\small{\abovedisplayshortskip=\abovedisplayskip}
\appto\footnotesize{\abovedisplayshortskip=\abovedisplayskip}

\usepackage{amsmath}
\usepackage{amsthm}
\usepackage{amssymb}
\usepackage{lineno}
\usepackage{mathtools}
\usepackage{booktabs}
\usepackage{setspace}
\usepackage{fancyhdr}
\usepackage{lastpage}
\usepackage{extramarks}
\usepackage{chngpage}
\usepackage{soul,color}
\usepackage{bm}
\usepackage{graphicx,float,wrapfig}
\usepackage{subcaption}
\usepackage{enumerate}
\usepackage{algorithm}
\usepackage{algpseudocode}
\usepackage{array}
\usepackage[colorlinks,linkcolor=blue,anchorcolor=green,citecolor=red]{hyperref}
\usepackage[mathscr]{euscript}

\usepackage{tikz}
\usetikzlibrary{shapes,
				arrows.meta,
				calc,
                math,
				positioning,
                patterns,
				decorations.pathreplacing}

\usepackage{pgfplots, pgfplotstable}

\newtheorem{definition}{Definition}
\newtheorem{lemma}{Lemma}
\newtheorem{proposition}{Proposition}
\newtheorem{corollary}{Corollary}
\newtheorem{theorem}{Theorem}

\newcommand{\PreserveBackslash}[1]{\let\temp=\\#1\let\\=\temp}
\newcolumntype{C}[1]{>{\PreserveBackslash\centering}p{#1}}
\newcolumntype{R}[1]{>{\PreserveBackslash\raggedleft}p{#1}}
\newcolumntype{L}[1]{>{\PreserveBackslash\raggedright}p{#1}}

\newcommand{\nn}{\mathbb{N}}
\newcommand{\rr}{\mathbb{R}}

\newcommand{\G}{\ensuremath{\mathcal{G}}}
\renewcommand{\H}{\ensuremath{\mathcal{H}}}

\newcommand{\emp}{\ensuremath{\emptyset}}

\newcommand{\F}{\ensuremath{\mathcal{F}}}

\newcommand{\A}{\ensuremath{\mathcal{A}}}
\newcommand{\B}{\ensuremath{\mathcal{B}}}

\newcommand{\convhull}{\mathsf{conv.hull}}

\newcommand{\cls}[1]{\overline{#1}}

\newcommand{\lsk}{\mathsf{LS}}
\newcommand{\ld}{\mathsf{LH}}

\newcommand{\dpr}[2]{{#1}{}^T{#2}}

\newcommand{\Rbar}{\overline{\rr}}

\newcommand{\dual}[2]{h_{#1}(#2)}

\newcommand{\lam}{\lambda}
\newcommand{\di}{\mathrm{d}}

\newcommand{\lm}[2]{\frac{\di{#1}}{\di{#2}}}

\newcommand{\vx}{\vec{x}}
\newcommand{\vy}{\vec{y}}
\newcommand{\vz}{\vec{z}}
\newcommand{\va}{\vec{a}}
\newcommand{\vb}{\vec{b}}
\newcommand{\vc}{\vec{c}}
\newcommand{\vv}{\vec{v}}

\newcommand{\vxs}{\vec{x}^*}

\newcommand{\llsb}{[\![}
\newcommand{\rrsb}{]\!]}

\newcommand{\veps}{\varepsilon}

\newcommand{\ignore}[1]{}

\newcommand{\eqn}{\begin{eqnarray}}
\newcommand{\uneq}{\end{eqnarray}}
\newcommand{\eqns}{\begin{eqnarray*}}
\newcommand{\uneqs}{\end{eqnarray*}}

\newcommand{\ra}{\rightarrow}

\newcommand{\Finf}{\ensuremath{F_{\rm{inf}}}}

\newcommand{\myprime}{\mkern-1.5mu\raise1.3ex\hbox{$\scriptstyle\prime$}}

\newcommand{\frpex}[1]{\mkern-4mu\raise2.5ex\hbox{$\scriptstyle{#1}$}}

\newcommand{\primeone}{\mkern-10mu\raise0.85ex\hbox{$\scriptstyle\prime$}}
\newcommand{\primetwo}{\mkern-9mu\raise0.85ex\hbox{$\scriptstyle\prime$}}
\newcommand{\primei}{\mkern-10mu\raise0.85ex\hbox{$\scriptstyle\prime$}}
\newcommand{\primeM}{\mkern-4mu\raise0.85ex\hbox{$\scriptstyle\prime$}}
\newcommand{\primeQ}{\mkern-6mu\raise0.85ex\hbox{$\scriptstyle\prime$}}
\newcommand{\primeR}{\mkern-6mu\raise0.85ex\hbox{$\scriptstyle\prime$}}
\newcommand{\primeL}{\mkern-7mu\raise0.85ex\hbox{$\scriptstyle\prime$}}

\newcommand{\vmu}{\vec{\mu}}

\newcommand{\pmeq}{\operatorname{\lower.2em\hbox{$\overset{\mathrm{pm}}{\sim}$}}}

\newcommand{\vdamu}{\vdm{\valpha}{\mu}}
\newcommand{\vdaa}{\vdm{\valpha}{|\valpha|}}
\newcommand{\vdbb}{\vdm{\vbeta}{|\vbeta|}}
\newcommand{\cwpf}[1]{{\odot}^{#1}}
\newcommand{\cwp}{\cwpf{n}}

\newcommand{\valpha}{\vec{\alpha}}
\newcommand{\vbeta}{\vec{\beta}}

\newcommand{\vtau}{\vec{\tau}}

\newcommand{\sfU}{\mathsf{U}}

\newcommand{\LZ}{L^{\mkern-1.5mu\raise0.4ex\hbox{\tiny{$\scriptstyle{\leftarrow}$}}}}
\newcommand{\FinfZ}{\Finf^{\mkern0mu\raise0.4ex\hbox{\tiny{$\scriptstyle{\leftarrow}$}}}}

\newcommand{\domc}[3]{\llsb {#1} #2 {#3}\hspace{0.1ex}\rrsb}

\newcommand{\Dins}[2]{\domc{\dpr{\vxs}{#1}}{#2}{0}}
\newcommand{\Diamu}[1]{\Dins{\vdamu}{#1}}
\newcommand{\Diaa}[1]{\Dins{\vdaa}{#1}}

\newcommand{\vdbnu}{\vdm{\vbeta}{\nu}}

\newcommand{\dhaus}{d_{\mathrm{H}}}
\newcommand{\cube}{\mathsf{Cube}}
\newcommand{\Unit}{{\mathrm{U}^n_{\scriptscriptstyle{\|\|_1}}}}
\newcommand{\vu}{\vec{u}}
\newcommand{\vlambda}{\vec{\lambda}}
\newcommand{\vomega}{\vec{\omega}}

\newcommand{\tensor}{\otimes}

\newcommand{\oneUnit}[1]{{\mathrm{U}^{#1}_{\scriptscriptstyle{\|\|_1}}}}
\newcommand{\Unitm}{\oneUnit{m}}

\newcommand{\trany}[1]{\mathrm{Mt}[\cdot,{#1}]}
\newcommand{\sfK}{\mathsf{K}}
\newcommand{\vf}{\vec{f}}
\newcommand{\vg}{\vec{g}}

\newcommand{\vdm}[2]{\frac{\di #1}{\di #2}}

\newcommand{\sfZ}{\mathsf{Z}}
\newcommand{\bbE}{\mathbb{E}}

\newcommand{\nvec}[1]{{\bm{#1}}}
\renewcommand{\vec}[1]{{\bm{#1}}}

\newcommand{\ctensor}{\mathbin{\otimes_{\mathrm{c}}}}

\title{A Product on Vector Measure Ranges, Lorenz Hulls and Zonoids}
\author{John P. Steinberger \and Zhe Zhang}
\date{}
\begin{document}
\begin{spacing}{1.1}
\maketitle
\begin{abstract}
A \emph{Lorenz skeleton} and \emph{Lorenz hull} are the range and the
convex hull of the range of an $\rr^n$-valued measure, respectively.
A Lorenz hull is also known as a \emph{zonoid}.
We extend the tensor product of centered zonoids introduced by Aubrun
and Lancien \cite{Aubrun2015Zonoids} and developed by Breiding et al.
\cite{Breiding2022TheZonoid} to Lorenz hulls without a centering restriction,
and further to Lorenz skeletons. We show that this product admits a
direct interpretation in terms of the elementary notion of a product measure.
Inclusion preservation and Hausdorff continuity, shown by Aubrun and Lancien
and Breiding et al., respectively, for centered zonoids, hold in this
more general setting as well.
\end{abstract}

\section{Introduction}
Let $\valpha=(\alpha_1,\ldots,\alpha_n)$ be a vector of finite signed
measures on a common measurable space $(S,\F)$. Its \emph{Lorenz
skeleton} and \emph{Lorenz hull} are, respectively,
\[
\lsk(\valpha)=\{\valpha(E):E\in\F\},\qquad
\ld(\valpha)=\convhull(\lsk(\valpha)).
\]
Both sets are compact, contain the origin, and are centrally symmetric
about $\valpha(S)/2$. The skeleton need not be convex; when the measure
is nonatomic, it is convex and hence equals the hull. Lorenz hulls
are zonoids, and their relation to descriptions of zonoids by integrals
of segments is discussed in \autoref{sec:zonoids}, with the equivalent
range characterization proved in Appendix~\ref{sec:equivalences}.

The ``Lorenz'' moniker enters into the scene via the Lorenz curve,
popular in the social sciences, that can be characterized as the lower
boundary of a two-dimensional Lorenz hull, more specifically the lower
boundary of the hull of a pair of nonnegative measures on a common space.

Aubrun and Lancien \cite{Aubrun2015Zonoids} and Breiding et al.
\cite{Breiding2022TheZonoid} developed tensor products for zonoids
centered at the origin. Inclusion preservation was proved by Aubrun
and Lancien \cite[Lemma~3.4]{Aubrun2015Zonoids}, and also by Breiding
et al. \cite[Proposition~3.4]{Breiding2022TheZonoid}; Hausdorff
continuity was established in \cite[Theorem~3.10(1)]{Breiding2022TheZonoid}.
We extend these properties to the setting considered here.
Here we allow the centers of Lorenz hulls to
be arbitrary and also retain the possibly nonconvex ranges themselves.
The construction uses product measures. If $\vbeta=(\beta_1,\ldots,\beta_m)$
is another vector measure, define
\[
\valpha\tensor\vbeta=(\alpha_i\times\beta_j)_{i,j}.
\]
On measurable rectangles this measure satisfies
$(\valpha\tensor\vbeta)(E\times F)=\valpha(E)\tensor\vbeta(F)$.
The question is whether its range, or the convex hull of its range,
depends only on the corresponding sets for the two factors.

We prove the stronger assertion that inclusion is preserved:
\begin{align*}
\lsk(\valpha)\subseteq\lsk(\valpha'),\quad
\lsk(\vbeta)\subseteq\lsk(\vbeta')
&\ \Longrightarrow\ 
\lsk(\valpha\tensor\vbeta)\subseteq
\lsk(\valpha'\tensor\vbeta'),\\
\ld(\valpha)\subseteq\ld(\valpha'),\quad
\ld(\vbeta)\subseteq\ld(\vbeta')
&\ \Longrightarrow\ 
\ld(\valpha\tensor\vbeta)\subseteq
\ld(\valpha'\tensor\vbeta').
\end{align*}
Applying these implications in both directions gives products on
Lorenz skeletons and on Lorenz hulls independent of the representing
measures. Both products are jointly continuous in the Hausdorff
distance. The skeleton argument proceeds through discrete
approximations; for hulls we also give a proof using support functions.

For skeletons, the approximate-inclusion estimate is uniform when
the source and target ranges have fixed bounds. It cannot in general
be made uniform over target ranges of arbitrary size, even with the
source measures fixed. A counterexample supplied by ChatGPT-6 and presented in
\autoref{sec:prodofskeletons} shows this already for scalar measures
on finite spaces. This does not obstruct exact inclusion preservation
or Hausdorff continuity.

Linear images of the tensor product give products induced by
multilinear maps. In particular, selecting the diagonal coordinates
when $m=n$ gives the coordinate-wise Lorenz product associated with
$(\alpha_1\times\beta_1,\ldots,\alpha_n\times\beta_n)$.
For hulls this product is commutative, associative, distributive over
Minkowski sums, and has identity $[\mathbf0,\mathbf1]$.
Its invertible elements are precisely the segments $[\mathbf0,\vx]$
whose endpoint has no zero coordinate.
We also relate the construction to the centered zonoid product and
extend the measure formulation to complex vector measures.

\section{\texorpdfstring{$\rr^n$}{R^n}-valued Measures and Derivatives}\label{sec:signedvectormeasures}

\noindent All Euclidean spaces below are finite-dimensional. The tensor
product is denoted by $\tensor$; the coordinate-wise product of measures
is denoted by $\times$.

We use standard measure theory, including the product measure theorem,
Tonelli's theorem, and dominated convergence; see
\cite{Durrett2005Probability,Rudin1987Real}. Euclidean spaces carry
their Borel sigma-algebras. We write $\Rbar=\rr\cup\{-\infty,+\infty\}$,
use $0\cdot(\pm\infty)=0$, and leave $\infty-\infty$ undefined.
The term integrable is used below whenever the specified integral
exists, possibly with an infinite value; finiteness assumptions are
stated explicitly where needed.

\begin{definition}\label{def:vvaluedmeasure}
    Let 
    $(S, \F)$ be a measurable space and let
    $V$ be a Euclidean space.
    A \emph{$V$-valued} measure on $(S, \F)$
    is a function $\vmu : \F \to V$ that is countably
    additive, i.e., such that
    
    $$
    \sum_{i}\vmu(A_i) = \vmu(\bigcup_i A_i)
    $$
    for all sequences $A_1$, $A_2$, $\dots$ of pairwise
    disjoint elements of $\F$.
\end{definition}

\noindent
\autoref{def:vvaluedmeasure} supplies the definition of an $\rr$-valued
measure, in particular.
Note that, by contrast with a positive measure, an $\rr$-valued measure
is defined as a finite-valued function. In particular, 
$\rr$-valued measures are not generalizations of positive measures,
but encompass an incomparable class of objects.
Moreover
the finiteness of $\rr$-valued measures
more strongly implies their boundedness,
as seen below in \autoref{lem:hahn}.

\begin{proposition}\label{prp:vmeasurecomponent}
    Let $(S, \F)$ be a measurable space and let
    $V$ be a Euclidean space with basis $e_1$, $\dots$, $e_n$.
    Then $\vmu : \F \ra V$ is a ($V$-valued) measure
    if and only if the component of $\vmu$ in the direction of each basis element $e_i$ is
    an $\rr$-valued measure.
\end{proposition}

\noindent
\autoref{prp:vmeasurecomponent} allows us to identify 
an $\rr^n$-valued measure $\valpha$
with an $n$-tuple $(\alpha_1, \ldots, \alpha_n)$ of 
$\rr$-valued measures via the standard basis.

\begin{lemma}
    \label{lem:hahn}
    \emph{(Hahn/Jordan decomposition \cite{Rudin1987Real})}
    A real-valued measure on a measurable space $(S, \F)$
    can be uniquely written as
    $\alpha^+ - \alpha^-$ for finite 
    positive measures $\alpha^+$, $\alpha^-$
    for which there exist disjoint $S_+$, $S_- \in \F$
    such that $\alpha^+ = 0$ outside of $S_+$ and
    $\alpha^- = 0$ outside of $S_-$.
\end{lemma}

\noindent
Since a positive measure is finite if
and only if it is bounded the positive measures
$\alpha^-$, $\alpha^+$ of \autoref{lem:hahn} are bounded,
and any $\rr$-valued measure $\alpha$ is bounded. 
By \autoref{prp:vmeasurecomponent},
all Euclidean-valued measures are 
also bounded.

For $\mu$ a $\sigma$-finite positive measure on a measurable space $(S, \F)$ and
$f : S \ra \rr$ a measurable function with respect to $\F$,
we recall that
\begin{equation}
    \label{eq:positivedecomp}
    \int_S f \di\mu \coloneqq \int_S f^+ \di\mu - \int_S f^- \di\mu
\end{equation}
with $f^+ \coloneqq \max(f, 0)$, $f^- \coloneqq \max(-f, 0)$,
where the integrals on the right-hand side of \eqref{eq:positivedecomp}
are guaranteed to exist, and where \eqref{eq:positivedecomp} exists
if and only if the difference does not take the form $\infty - \infty$,
in which case we say that $f$ is \emph{integrable} with respect to $\mu$.

If $\mu$ and $f$ are both bounded, in particular, then $f$ is
integrable with respect to $\mu$.

For $\alpha$ an $\rr$-valued measure on $(S, \F)$, in turn,
we define
\begin{equation}
    \label{eq:decomp}
    \int_S f \di\alpha \coloneqq \int_S f \di\alpha^+ - \int_S f \di\alpha^-
\end{equation}
for measurable $f : S \ra \rr$,
with $\alpha^+$, $\alpha^-$ as in
\autoref{lem:hahn},
with the integral existing if and only if 
the right-hand side of \eqref{eq:decomp} exists,
in which case we again say that $f$ is \emph{integrable} (with respect to $\alpha$).
One can observe that any bounded (measurable) $f : S \ra \rr$ is 
integrable with respect to any $\rr$-valued measure since the boundedness
of the measure implies the (existence and) finiteness of the integrals involved.

Completing these definitions, we define
\begin{equation}
\int_S f \di\valpha \coloneqq \Big(\int_S f \di\alpha_1, \ldots, \int_S f \,\di\alpha_n\Big) \label{eq:man}
\end{equation}
for an $\rr^n$-valued measure $\valpha = (\alpha_1, \dots, \alpha_n)$ on $(S, \F)$
for measurable $f$, with existence subject to the existence of each coordinate,
possibly as an extended real value.

In practice we like to evaluate the integrals
on the right-hand side of \eqref{eq:man} with respect
to a single “common denominator” measure while 
appealing to the Radon-Nikodym derivative, as outlined
in the next definition and lemmas.

\begin{definition}
    A positive measure $\mu$ on $(S, \F)$ \emph{dominates}
    an $\rr$-valued measure $\alpha$ on $(S, \F)$ if 
    $\mu(A) = 0 \implies \alpha(A) = 0$ for all $A \in \F$.
    A positive measure $\mu$ on $(S, \F)$ \emph{dominates} 
    an $\rr^n$-valued measure $\valpha$ on $(S, \F)$
    if $\mu$ dominates each coordinate of $\valpha$.
\end{definition}

\begin{lemma}\label{lem:radonnikodymcomplex}
    \emph{(6.10, Rudin \cite{Rudin1987Real})}
    Let $\alpha$ be an $\rr$-valued measure on $(S, \F)$ and
    let $\mu$ be a $\sigma$-finite positive measure on $(S, \F)$.
    If $\mu$ dominates $\alpha$ then there exists a measurable function
    $g: S \to \rr$ such that 
    $$
    \alpha(A) = \int_A g \,\di\mu
    $$
    for all $A \in \F$. Moreover, if $h$ is another such function, then $h = g$ $\mu$-almost everywhere.
\end{lemma}

\noindent
Correspondingly, we write
$$
\lm{\alpha}{\mu}
$$
for the function $g$ in \autoref{lem:radonnikodymcomplex},
defined $\mu$-almost everywhere,
referred to as \emph{the Radon-Nikodym derivative of $\alpha$ with respect to $\mu$}.
We fix measurable representatives. Identities between Radon--Nikodym
derivatives are understood almost everywhere with respect to the denominator
measure. We have:

\begin{lemma}\label{lem:baseofintoverRmeasure}
    \emph{(\cite{Rudin1987Real})}
    If $\alpha$ is an $\rr$-valued measure
    and $\mu$ is a $\sigma$-finite positve measure dominating $\alpha$, then
    \begin{equation}
        \label{eq:baseofintoverRmeasure}
        \int_S f \,\di\alpha = \int_S f \cdot \lm{\alpha}{\mu} \,\di\mu
    \end{equation}
    for any measurable $f: S \to \rr$, with either integral existing
    if and only if the other exists.
\end{lemma}

\noindent
By extension to \autoref{lem:baseofintoverRmeasure},
\begin{equation}
    \int_{S} f \,\di\valpha = \int_S f \cdot \vdamu \,\di\mu
\end{equation}
where
$$
\vdamu \coloneqq \Big(\lm{\alpha_1}{\mu}, \lm{\alpha_2}{\mu}, \ldots, \lm{\alpha_n}{\mu}\Big) : S \to \rr^n
$$
for $\valpha = (\alpha_1, \dots, \alpha_n)$ an $\rr^n$-valued measure
on $(S, \F)$ and $\mu$ a finite positive measure dominating $\valpha$
and where
$$
\int_S \vec{f} \,\di\nu \coloneqq \Big(\int_S f_1 \,\di\nu, \int_S f_2 \,\di\nu, \ldots, \int_S f_n \,\di\nu\Big)
$$
for $\vec{f} = (f_1, f_2, \dots, f_n): S \to \rr^n$ measurable,
$\nu$ a $\sigma$-finite positive measure on $(S, \F)$.

We put
$$
|\alpha| \coloneqq \alpha^+ + \alpha^-
$$
(the so-called \emph{total variation} of $\alpha$)
for an $\rr$-valued measure $\alpha$ on $(S, \F)$,
where $\alpha^+$, $\alpha^-$ are the
elements of the Hahn/Jordan decomposition,
and
$$
|\valpha| \coloneqq \sum_{i=1}^n |\alpha_i|
$$
for an $\rr^n$-valued measure $\valpha = (\alpha_1, \dots, \alpha_n)$
on $(S, \F)$.
Then 
$|\alpha|$, $|\valpha|$ are $\sigma$-finite (in fact, finite) 
and dominate $\alpha$,
$\valpha$, respectively, as is easy to see.
For the second claim, more precisely, one can note that
$$
|\alpha(A)| = |\alpha^+(A) - \alpha^-(A)| \leq |\alpha|(A),
$$
and that
$$
\|\valpha(A)\|_1 = \sum_{i=1}^n |\alpha_i(A)| \leq |\valpha|(A)
$$
for all $A \in \F$, implying $|\alpha|(A) = 0 \implies |\alpha(A)| = 0$,
$|\valpha|(A) = 0 \implies \|\alpha(A)\|_1 = 0$ for all $A \in \F$.

Our work requires \autoref{prp:aboutvdaa} below,
for which the following two lemmas come into play:

\begin{lemma}\label{lem:tvisintegralofab}
    \emph{(\cite{Rudin1987Real}, Chapter 6)}
    If $\alpha$ is an $\rr$-valued measure,
    $$
    \Big|\lm{\alpha}{\mu}\Big| = \lm{|\alpha|}{\mu}
    $$
    for any $\sigma$-finite positive measure $\mu$ dominating $\alpha$.
\end{lemma}

\begin{lemma}\label{lem:almosteverywhererange}
    \emph{(Theorem 1.40, Rudin \cite{Rudin1987Real})}
    Let $\mu$ be a finite positive measure on $(S, \F)$ and let $f: S \to \rr$ be such that $\int_S |f| \,\di\mu < \infty$.
    Let $E$ be a closed subset of $\rr$. If
    $$
    \int_A f \,\di\mu \in E\cdot\mu(A)
    $$
    for every $A \in \F$ then $f \in E$ $\mu$-almost everywhere on $S$.
\end{lemma}

\begin{proposition}\label{prp:aboutvdaa}
    Let $\valpha = (\alpha_1, \ldots, \alpha_n)$ be an $\rr^n$-valued measure on $(S, \F)$.
    Then 
    $$
    \bigg\|\vdm{\valpha}{|\valpha|}\bigg\|_1 = 1
    $$
    $|\valpha|$-almost everywhere.
\end{proposition}
\begin{proof}
    We have
    $$
    \int_A \bigg|\lm{\alpha_i}{|\valpha|}\bigg| \,\di|\valpha| = \int_A \lm{|\alpha_i|}{|\valpha|} \,\di|\valpha| = |\alpha_i|(A)
    $$
    by \autoref{lem:tvisintegralofab} and the definitional property of
    the Radon-Nikodym derivative, respectively, from which
    $$
    \int_A \bigg\|\vdaa\bigg\|_1 \,\di|\valpha| = \int_A \left(\sum_{i=1}^{n} \bigg|\lm{\alpha_i}{|\valpha|}\bigg|\right) \,\di|\valpha| = \sum_{i=1}^n|\alpha_i|(A) = |\valpha|(A)
    $$
    for all $A \in \F$, with the application of linearity being
    justified by the existence of the constituent integrals. The 
    conclusion follows by \autoref{lem:almosteverywhererange}
    applied with $E = \{1\}$. 
\end{proof}

\section{Tensor Products of Vector Measures}\label{sec:tensormeasures}

\noindent
For measurable spaces $(S, \F)$, $(T, \G)$, we let
$$
(S, \F) \times (T, \G)
$$
stand for the measurable space $(S\times T, \H)$
where $\H$ is the $\sigma$-algebra of $S\times T$
generated by the rectangles $A \times B$, $A\in \F$, $B \in \G$.
The following lemma 
asserts the existence of
the classical notion of a product measure
for positive measures:

\begin{lemma}\label{lem:prodmeasure}
    Let $\mu$, $\nu$ be $\sigma$-finite positive measures on $(S, \F)$, $(T, \G)$, respectively.
    There exists a unique $\sigma$-finite positive measure, written $\mu \times \nu$, on $(S, \F) \times (T, \G)$
    such that
    $$
    (\mu \times \nu)(A \times B) = \mu(A)\nu(B)
    $$
    for all $A \in \F$, $B \in \G$.
\end{lemma}

\noindent
To extend this result to real-valued
measures we need the following technical 
lemma that can be proved by the Hahn-Jordan decomposition (\autoref{lem:hahn})
and by an anolog of the lemma for positive measures
(\cite{Durrett2005Probability}, Appendix A.1):

\begin{lemma}\label{lem:rmeasureagree}
    Let $\mathcal{P}$ be a collection of subsets of $S$ 
    such that $\mathcal{P}$ is closed under intersection and such that $S \in \mathcal{P}$.
    Then Euclidean-valued measures on $(S, \sigma(\mathcal{P}))$
    agree on $\sigma(\mathcal{P})$ if they agree on $\mathcal{P}$.
\end{lemma}

\noindent
We now have:

\begin{lemma}\label{lem:rprodmeasure}
    Let $\alpha$, $\beta$ be $\rr$-valued measures on $(S, \F)$, $(T, \G)$, respectively.
    There exists a unique $\rr$-valued measure, 
    that we can write $\alpha \times \beta$ without conflict
    with \autoref{lem:prodmeasure}, on $(S, \F) \times (T, \G)$
    such that
    $$
    (\alpha \times \beta)(A \times B) = \alpha(A)\beta(B)
    $$
    for all $A \in \F$, $B \in \G$.
\end{lemma}
\begin{proof}
    Let $\alpha = \alpha^+ - \alpha^-$, $\beta = \beta^+ - \beta^-$
    be the Hahn-Jordan decompositions of $\alpha$, $\beta$.
    Since each of the components is a finite positive measure
    we can form the linear combination 
    of finite positive measures
    $$
    (\alpha^+\times\beta^+) + (\alpha^-\times\beta^-) - (\alpha^+\times\beta^-) - (\alpha^-\times\beta^+)
    $$
    by \autoref{lem:prodmeasure},
    that is
    an $\rr$-valued measure on $(S, \F) \times (T, \G)$.
    It easy to see that this measure satisfies the
    property required of $\alpha\times\beta$. Uniqueness
    of $\alpha\times\beta$ follows by \autoref{lem:rmeasureagree}.
\end{proof}

\noindent

The existence 
of an analogous tensor product
of Euclidean-valued measures is immediately deduced via basis decomposition: 

\begin{proposition}
    Let $\valpha$, $\vbeta$ be $V$- and $W$-valued measures
    respectively where $V$ and $W$ are Euclidean spaces,
    defined on measurable spaces 
    $(S, \F)$, $(T, \G)$ respectively.
    Then there exists a unique  
    $V \tensor W$-valued
    measure $\valpha \tensor \vbeta$
    on 
    $(S, \F) \times (T, \G)$
    such that
    $$
    (\valpha \tensor \vbeta)(A \times B) = \valpha(A) \tensor \vbeta(B)
    $$
    for all $A \in \F$, $B \in \G$.
\end{proposition}
\begin{proof}
    Let $v_1, \dots, v_{n}$ be a basis of $V$,
    $w_1, \dots, w_{m}$ be a basis of $W$.
    Then $v_1 \tensor w_1, \dots, v_{n} \tensor w_{m}$ is a basis of $V\tensor W$.
    Suppose $\valpha = (\alpha_1, \dots, \alpha_{n})$, $\vbeta = (\beta_1, \dots, \beta_{m})$
    with respect to the bases $v_1, \dots, v_{n}$, $w_1, \dots, w_{m}$,
    so that $\alpha_i$, $\beta_j$ are $\rr$-valued measures for each $i$, $j$ by \autoref{prp:vmeasurecomponent}.
    Assuming the existence of $\valpha \tensor \vbeta$,
    and with $\gamma_{i,j}$ denoting the component along $v_i \tensor w_j$ 
    of $\valpha \tensor \vbeta$,
    we have
    \begin{align*}
        \sum_{i,j} \gamma_{i,j}(A \times B) (v_i \tensor w_j) &= (\valpha \tensor \vbeta)(A \times B) \\
            &= \valpha(A) \tensor \vbeta(B)\phantom{\sum_{i,j}} \\
            &= \sum_{i, j} \alpha_i(A)\beta_j(B) (v_i \tensor w_j)
    \end{align*}
    for all $A \in \F$, $B \in \G$, implying
    $$
    \gamma_{i,j}(A \times B) = \alpha_i(A)\beta_j(B)
    $$
    for all $A \in \F$, $B \in \G$, 
	implying $\gamma_{i,j} = \alpha_i \times \beta_j$ for all $i$, $j$ by \autoref{lem:rmeasureagree}.
    Moreover, 
    setting $\gamma_{i,j} = \alpha_i \times \beta_j$ for all $i$, $j$
    by \autoref{lem:rprodmeasure}
    makes $\valpha \tensor \vbeta$ into a measure by \autoref{prp:vmeasurecomponent},
    satisfying
    \begin{align*}
        (\valpha \tensor \vbeta)(A \times B) &= \sum_{i,j} \gamma_{i,j}(A \times B) (v_i \tensor w_j) \\
        &= \sum_{i, j} \alpha_i(A)\beta_j(B) (v_i \tensor w_j) \phantom{\sum_{i,j}}  \\
        &= \valpha(A) \tensor \vbeta(B)
    \end{align*}
    completing the proof of existence.
\end{proof}

\noindent
In the remainder of the paper we will prefer to work with
$\rr^n$-valued measures 
over Euclidean-valued 
measures, to afford the convenience of standard norms such
as the 1-norm $\|\cdot\|_1$ on the range as well as the
existence of a canonical correspondence between the vector-valued 
measure and a vector of real-valued measures.
Results extend to the more general Euclidean 
case in the obvious manner.

\begin{lemma}\label{lem:fubinivec}
    Let $\mu$, $\nu$ be $\sigma$-finite positive measures on $(S, \F)$, $(T, \G)$, respectively.
    Let $\vec{f}: S \times T \to \rr^n$ be a measurable function 
    with respect to the $\sigma$-algebra of $(S, \F) \times (T, \G)$. 
    If any of 
    $$
    \int_S \int_T \|\vec{f}(s, t)\|_1 \,\nu(\di t) \,\mu(\di s), 
    $$
    $$
    \int_T \int_S \|\vec{f}(s, t)\|_1 \,\mu(\di s) \,\nu(\di t)
    $$
    or
    $$
    \int_{S \times T} \|\vec{f}\|_1 \,\di(\mu \times \nu)
    $$
    is finite 
    then the functions $\vec{f}_T: S \to \rr^n$, $\vec{f}_S: T \to \rr^n$ defined by
    $$
    \vec{f}_T(s) = \int_T \vec{f}(s, t) \,\nu(\di t), \quad \vec{f}_S(t) = \int_S \vec{f}(s, t) \,\mu(\di s)
    $$
    wherever these section integrals exist, and set to zero on the
    respective measurable null sets where they do not exist, are integrable
    over $\mu$, $\nu$, respectively, and
    $$
    \int_S \vec{f}_T \,\di\mu 
    = \int_{S \times T} \vec{f} \,\di(\mu \times \nu) 
    = \int_T \vec{f}_S \,\di\nu.
    $$
    
\end{lemma}

\begin{corollary}\label{col:fubinitensor}
    Let $\mu$, $\nu$ be $\sigma$-finite positive measures on $(S, \F)$, $(T, \G)$, respectively.
    Let $\vec{f}: S \ra \rr^n$, $\vec{g}: T \ra \rr^m$ be measurable functions with respect to $\F$, $\G$, respectively.
    Then $\vec{f} \tensor \vec{g}$ is measurable from $(S, \F) \times (T, \G)$ to $\rr^n \tensor \rr^m$, and
    $$
    \int_S \vec{f} \,\di\mu \tensor \int_T \vec{g} \,\di\nu = \int_{S \times T} \vec{f} \tensor \vec{g} \,\di(\mu \times \nu)
    $$
    if $\int_S \|\vec{f}\|_1 \,\di\mu$, $\int_T \|\vec{g}\|_1 \,\di\nu$ are finite.
\end{corollary}
\begin{proof}
    The fact that $\vec{f} \tensor \vec{g}$ is measurable can be seen by
    coordinate-wise application of standard arguments on the measurability
    of the product of real-valued functions.
    Then,
    \begin{align*}
        \infty 
        & > \int_S \|\vec{f}\|_1 \,\di\mu \int_T \|\vec{g}\|_1 \,\di\nu \\
        & = \int_S \int_T \|\vec{f}\|_1(s) \|\vec{g}\|_1(t) \,\nu(\di t) \,\mu(\di s) \\
        & = \int_S \int_T \|\vec{f}(s) \tensor \vec{g}(t)\|_1 \,\nu(\di t) \,\mu(\di s),
    \end{align*}
    which implies
    $$
    \int_S \int_T \vec{f}(s) \tensor \vec{g}(t) \,\nu(\di t) \,\mu(\di s) = \int_{S \times T} \vec{f} \tensor \vec{g} \,\di(\mu \times \nu)
    $$
    by \autoref{lem:fubinivec}, from which the conclusion follows by linearity.
\end{proof}

\noindent
The following lemma, necessary for \autoref{prp:RNderivoftensor} below,
can be compared to \autoref{lem:rmeasureagree},
for which it provides an extension to integration:

\begin{lemma}\label{lem:vfisrdderivofvalpha}
    Let $\valpha$ be 
    an $\rr^n$-valued
    measure on $(S, \F)$ 
    and let $\mu$ be a $\sigma$-finite positive measure on the same measurable space.
    Let $\mathcal{P}$ be a collection of subsets of $S$ 
    such that $\mathcal{P}$ is closed under intersection and such that $\sigma(\mathcal{P}) = \F$, $S \in \mathcal{P}$.
    If $\vec{f}: S \to \rr^n$ is measurable and is such that 
    $$
    \int_A \vec{f} \,\di\mu = \valpha(A)
    $$
    for all $A \in \mathcal{P}$, then $\mu$ dominates $\valpha$ and 
    $$
    \lm{\valpha}{\mu} = \vec{f}
    $$
    $\mu$-almost everywhere on $S$.
\end{lemma}
\begin{proof}
    This is a mechanical consequence of \autoref{lem:rmeasureagree}
    and of the uniqueness of the Radon-Nikodym derivative.
\end{proof}

\begin{proposition}\label{prp:RNderivoftensor}
    Let $\valpha$, $\vbeta$ be $\rr^n$-, $\rr^m$-valued measures on $(S, \F)$, $(T, \G)$ 
    with dominating $\sigma$-finite positive measures $\mu$, $\nu$, respectively. 
    Then $\mu \times \nu$ dominates $\valpha \tensor \vbeta$ and 
    $$
    \vdm{(\valpha \tensor \vbeta)}{(\mu \times \nu)}(s, t) = \vdamu(s) \tensor \vdbnu(t)
    $$
    for $(\mu\times\nu)$-almost every $(s,t)\in S\times T$.
\end{proposition}
\begin{proof}
    Since the set of rectangles
    $$
    \{A \times B \,:\, A \in \F, B \in \G\}
    $$
    satisfies all conditions of the collection $\mathcal{P}$ in \autoref{lem:vfisrdderivofvalpha}
    with respect to the measurable space $(S, \F) \times (T, \G)$, it suffices to show that
    $$
    \int_{A \times B} \vdamu(s) \tensor \vdbnu(t) \,(\mu \times \nu)(\di(s, t)) = (\valpha \tensor \vbeta)(A \times B)
    $$
    for all $A \in \F$, $B \in \G$. However, 
    \begin{align*}
        & \quad \ \int_{A \times B} \vdamu(s) \tensor \vdbnu(t) \,(\mu \times \nu)(\di(s, t)) \\
        & = \int_A \vdamu \,\di\mu \tensor \int_B \vdbnu \,\di\nu
        = \valpha(A) \tensor \vbeta(B) = (\valpha \tensor \vbeta)(A \times B)
    \end{align*}
    for all $A \in \F$, $B \in \G$, where the first equality follows by \autoref{col:fubinitensor}
    (since 
    $$
    \int_A \Big\|\lm{\valpha}{\mu}\Big\|_1 \,\di\mu = |\valpha|(A), 
    \quad \int_B \Big\|\lm{\vbeta}{\nu}\Big\|_1 \,\di\nu = |\vbeta|(B)
    $$
    are finite values by similar computation as in the proof of \autoref{prp:aboutvdaa}).
\end{proof}

\section{Hausdorff Distance}\label{sec:hausdorff}
Nonempty sets and nonnegative approximation parameters are understood below.

\begin{definition}\label{def:hausdorffdistance}
    Let $p$ denote a norm on $\rr^n$. 
    Let 
    $$
    d(\vx, A) \coloneqq \inf_{\vy \in A} p(\vx - \vy)
    $$
    for all $\vx \in \rr^n$, $A \subseteq \rr^n$.
    The \emph{Hausdorff distance} $\dhaus(A, B)$ induced by $p$ between subsets $A$ and $B$ of $\rr^n$ is defined by
    $$
    \dhaus(A, B) = \max\{\sup_{\vx \in A} d(\vx, B), \sup_{\vy \in B} d(\vy, A)\}
    $$
    for all $A$, $B \subseteq \rr^n$.
\end{definition}

\noindent

\begin{definition}
    $A \subseteq \rr^n$ is \emph{$\veps$-almost included
    in $B \subseteq \rr^n$ [with
    respect to norm $p$]},
    written $A \subseteq_\veps B$, if 
    $$
    \sup_{\vx \in A} d(\vx, B) \leq \veps.
    $$
\end{definition}

\noindent

\begin{proposition}\label{prp:vepsincludiffforxexistsy}
    If $A, B \subseteq \rr^n$ and $B$ is compact then
    $A \subseteq_\veps B$ if and only if for any $\vx \in A$ there exists $\vy \in B$ such that $p(\vx - \vy) \leq \veps$.
\end{proposition}

\begin{proposition}\label{prp:includiffvepsincludforeveryveps}
    If $A, B \subseteq \rr^n$ and $B$ is compact then
    $A \subseteq B$ if and only if $A \subseteq_\veps B$ for any $\veps > 0$.
\end{proposition}

\noindent
For nonempty sets, almost inclusion is transitive with addition of errors:
$A\subseteq_a B$, $B\subseteq_b C$ imply $A\subseteq_{a+b}C$.
This follows from the triangle inequality, taking infima and then suprema.
Also $\dhaus(A,B)\leq a$ exactly when both directed almost inclusions hold.

\section{Inclusion Preservation of Lorenz Skeletons}\label{sec:prodofskeletons}
\noindent Throughout this section, distances use $p=\|\cdot\|_1$.
Write $[k]=\{1,\ldots,k\}$ and let $\Unit$, $\Unitm$ denote the
$1$-norm unit spheres in $\rr^n$, $\rr^m$, respectively.

\begin{definition}
	Let $\valpha$ be an $\rr^n$-valued measure on $(S, \F)$.
	The \emph{Lorenz skeleton} of $\valpha$, denoted
	$\lsk(\valpha)$, is the range of $\valpha$:
	$$
	\lsk(\valpha) := \{\valpha(A): A \in \F\} \subseteq \rr^n.
	$$
	The \emph{Lorenz hull} of $\valpha$, denoted $\ld(\valpha)$, is 
	the convex hull of the Lorenz skeleton $\lsk(\valpha)$.
\end{definition}

Both the skeleton and the hull contain zero and are centrally symmetric
about $\valpha(S)/2$, since $\valpha(S\setminus A)=\valpha(S)-\valpha(A)$.
\begin{lemma}\label{lem:codexcompact}
The range of a finite-dimensional real vector measure is compact.
Consequently its Lorenz hull is compact and convex.
\end{lemma}
\begin{proof}
The compactness part of the vector-measure range theorem
\cite{Halmos1948TheRange} gives the first assertion, without any
nonatomicity assumption. The convex hull of a compact subset of
$\rr^n$ is compact by Caratheodory's theorem \cite{Rudin1976Principles}.
\end{proof}
For disjoint measurable spaces $(S_1,\F_1)$, $(S_2,\F_2)$, define
\[
\F_1\oplus\F_2=\{A_1\cup A_2:A_1\in\F_1,\ A_2\in\F_2\}.
\]
For measures $\valpha_1$, $\valpha_2$ on these spaces,
$\valpha_1\oplus\valpha_2$ is the measure on their disjoint union given by
$(\valpha_1\oplus\valpha_2)(E)=\valpha_1(E\cap S_1)+\valpha_2(E\cap S_2)$.
Tagged copies are understood if the underlying sets are not disjoint.
Its skeleton is the Minkowski sum of the two skeletons.

\noindent

If $M$ is an nonnegative number such that $\lsk(\valpha) \subseteq \cube^n(M)$
for an $\rr^n$-valued measure $\valpha = (\alpha_1, \dots, \alpha_n)$, where 
$$
\cube^n(M) \coloneqq \{(z_1, \dots, z_n) \in \rr^n \,:\, |z_i| \leq M, 1 \leq i \leq n\},
$$
then
$$
|\alpha_i|(S) = \alpha_i^+(S) + \alpha_i^-(S) \leq 2 \max\{\alpha_i^+(S), \alpha_i^-(S)\} \leq 2 M
$$
for each $i$, $1 \leq i \leq n$, and
\begin{equation}\label{eqn:boundoftvwrt1norm}
    |\valpha|(S) = \sum_{i = 1}^{n} |\alpha_i|(S) \leq 2 n M
\end{equation}
as we shall make of use later.

\noindent\textsc{Atoms.} We recall the atomic decomposition used in
the approximation argument below.

\begin{definition}\label{def:nonatomic}
    Given an $\rr$-valued measure $\alpha$ on $(S, \F)$,
    a set $A \in \F$ is an \emph{atom} of $\alpha$ if $\alpha(A) \neq 0$
    and if for any $B \in \F$ either $\alpha(A \cap B) = \alpha(A)$ or $\alpha(A \cap B) = 0$.
    An  $\rr$-valued measure is \emph{non-atomic} if it has no atoms,
    and an $\rr^n$-valued measure $\valpha$ is \emph{non-atomic} if
    each of its components is non-atomic.
\end{definition}

\noindent
In fact, we do not define the concept of
an ``atom'' of an $\rr^n$-valued measure $\valpha$,
though one 
might have defined 
these as the atoms 
of $|\valpha|$ since, as seen below, 
an $\rr^n$-valued measure 
$\valpha$
is non-atomic if and only if $|\valpha|$ is non-atomic.

Non-atomic $\rr^n$-valued measures have the following striking property,
by which the convexity of their Lorenz skeletons is established:

\begin{lemma}\label{lem:nonatomicvmisconvex}
    \emph{(Lyapunov \cite{Lyapunov1940Ser}, Halmos \cite{Halmos1948TheRange})}
    Let $\valpha$ be a non-atomic $\rr^n$-valued measure on $(S, \F)$.
    Then for any $A \in \F$ there exists a map $\varphi$ from $[0, 1]$ to subsets of $A$ in $\F$
    such that $\varphi(0) = \emp$, $\varphi(1) = A$,
    $\varphi(a) \subseteq \varphi(b)$ if $a < b$,
    and $\valpha(\varphi(\lam)) = \lam\valpha(A)$ for all $\lam \in [0, 1]$.
\end{lemma}

\noindent

\begin{proposition}\label{prp:atomsofvecmeasures}
    For an $\rr$-valued measure $\alpha$ and an $\rr^n$-valued measure $\valpha = (\alpha_1, \ldots, \alpha_n)$
    on $(S, \F)$, the following properties hold:
    \begin{enumerate}[(i)]
        \item Atoms of $\alpha$ are atoms of $|\alpha|$, and vice versa;
        \item If $\vec{f}: S \to \rr^m$ is measurable and if $A$ is an atom of $\alpha$
            then there exists $\vc \in \rr^m$ such that $\vec{f}(s) = \vc$ 
            for $s \in A$ $|\alpha|$-almost everywhere;
        \item Each atom $A$ of $|\valpha|$ is an atom of at least
            one of the $\alpha_i$'s, and, specifically, is an atom of 
            $\alpha_j$ if and only if $|\alpha_j|(A) > 0$;
        \item For $i \in [n]$, each atom of $\alpha_i$ contains an atom of $|\valpha|$. 
    \end{enumerate}
\end{proposition}
\begin{proof}
    For (i), use a Hahn decomposition of $A$. If both its positive
    and negative parts have nonzero variation, $A$ cannot be an atom of
    either measure; otherwise the assertion reduces to a positive measure.
    For (ii), each real coordinate of a measurable function on an atom is
    almost surely constant: its rational sublevel sets have either zero or
    full measure. Intersect the finitely many full-measure sets.
    For (iii), $|\alpha_j|\leq|\valpha|$ implies that every subset of an
    atom of $|\valpha|$ has either zero or full $|\alpha_j|$-measure.
    Apply (i) and sum over $j$.
    For (iv), fix an atom $B$ of $\alpha_i$. If the restriction of
    $|\valpha|$ to $B$ had no atoms, split $B$ into two measurable pieces
    of equal $|\valpha|$-measure. Exactly one piece has the full
    $|\alpha_i|$-measure of $B$, by (i). Repeat on that piece to obtain
    nested $C_k\subseteq B$ with
    \[
    |\valpha|(C_k)=2^{-k}|\valpha|(B),\qquad
    |\alpha_i|(C_k)=|\alpha_i|(B)>0.
    \]
    This contradicts $|\alpha_i|(C_k)\leq|\valpha|(C_k)$.
\end{proof}

\noindent
Parts (iii) and (iv) of \autoref{prp:atomsofvecmeasures} imply that
$\valpha$ is non-atomic if and only if $|\valpha|$ is non-atomic.

Given an $\rr^n$-valued measure $\valpha$,
define two atoms $A_1$, $A_2$ of $|\valpha|$ to be \emph{related}
if $|\valpha|(A_1 \cap A_2) > 0$.
This defines an equivalence relation on the atoms.
There are at most countably many different equivalence classes since $|\valpha|$ is bounded.

The indexing families below may be finite or empty, with sums and unions
interpreted over the actual family. Let $\{A_i\}_{i \in \nn}$ be a collection of mutually disjoint (wlog)
atoms from each equivalence class.
Let
$$
A = \bigcup_{i = 1}^\infty A_i, \quad B = S \backslash A.
$$
The restriction of 
$|\valpha|$ to $\{E \cap B \,:\, E \in \F\}$
is then non-atomic,
and, by part (iv) of \autoref{prp:atomsofvecmeasures}, 
the restriction of $\valpha$ to the same set is also non-atomic.
We call this restriction of $\valpha$ to $B$ \emph{the non-atomic part of $\valpha$}
(with respect to the particular set $\{A_i\}_{i \in \nn}$ of chosen atoms),
and call the restriction of $\valpha$ to $\{E \cap A \,:\, E \in \F\}$ \emph{the purely atomic part of $\valpha$}.
We note that 
$$
\valpha(A) = \sum_{i=1}^\infty \valpha(A_i)
$$
and 
$$
\valpha(B) = \valpha(S) - \valpha(A)
$$
are independent of the choices of $\{A_i\}_{i \in \nn}$ by definition of atoms.
Lastly, we say that that $\valpha$ is \emph{purely atomic} if $|\valpha|(B) = 0$.

\noindent\textsc{Discrete approximation.} We first prove the
approximate-inclusion estimate for discrete measures, then approximate
general measures and their products by discrete ones.

\begin{proposition}\label{prp:discreteskeletonproductalmostinclusion}
    Let $\valpha$, $\valpha'$ be $\rr^n$-valued measures on $(S, \F)$, $(S', \F')$, respectively, 
    and let $\vbeta$, $\vbeta'$ be $\rr^m$-valued measures on $(T, \G)$, $(T', \G')$, respectively,
    where all four measurable spaces are countable sets with their power set $\sigma$-algebras
    such that $\lsk(\valpha) \subseteq_\delta \lsk(\valpha')$, 
    $\lsk(\vbeta) \subseteq_\delta \lsk(\vbeta')$ and
    $\lsk(\valpha)$, $\lsk(\valpha') \subseteq \cube^n(M_1)$, 
    $\lsk(\vbeta)$, 
    $\lsk(\vbeta') \subseteq \cube^m(M_2)$ for some values $\delta$, $M_1$,
    $M_2 \geq 0$. Then
    $$
    \lsk(\valpha \tensor \vbeta) \subseteq_\veps \lsk(\valpha' \tensor \vbeta')
    $$
    for $\veps \geq 2\delta (n M_1 + m M_2)$.
\end{proposition}

\begin{proof}
    Derivatives with respect to total variation are chosen bounded everywhere
    by setting them to zero on exceptional null sets.
    One can first note that all functions in this proof are measurable since the underlying spaces are discrete.
    If $\lsk(\valpha) \subseteq_\delta \lsk(\valpha')$, then, since Lorenz skeletons are compact sets,
    \autoref{prp:vepsincludiffforxexistsy} implies that
    for any $f: S \to \{0, 1\}$
    there exists a $f': S' \to \{0, 1\}$ such that
    $$
    \bigg\|\int_S f \,\di\valpha - \int_{S'} f' \,\di\valpha'\bigg\|_1 \leq \delta.
    $$
    Thus, for any $h: S \times T \to \{0, 1\}$ one can construct $h': S' \times T \to \{0, 1\}$ 
    such that
    \begin{equation}\label{eqn:tvhtdalphahtpdalphapldelta00}
        \bigg\|\int_S h_t \,\di\valpha - \int_{S'} h'_t \,\di\valpha'\bigg\|_1 \leq \delta
    \end{equation}
    for all $t \in T$, where $h_t: S \to \{0, 1\}$, $h'_t: S' \to \{0, 1\}$ are defined by
    $$
    h_t(s) = h(s, t), \quad h'_t(s') = h'(s', t)
    $$
    for all $s \in S$, $s' \in S'$.
    {\allowdisplaybreaks Then
    \begin{align*}
        & \quad \ \bigg\|\int_{S \times T} h \,\di(\valpha \tensor \vbeta) - \int_{S' \times T} h' \,\di(\valpha' \tensor \vbeta)\bigg\|_1 \\
        & = \bigg\|\int_T \int_S h(s, t) (\vdaa(s) \tensor \vdbb(t)) \,|\valpha|(\di s) \,|\vbeta|(\di t) \\
        & \qquad \qquad \qquad - \int_T \int_{S'} h'(s', t) (\vdm{\valpha'}{|\valpha'|}(s') \tensor \vdbb(t)) \,|\valpha'|(\di s') \,|\vbeta|(\di t)\bigg\|_1 \\
        & = \bigg\|\int_T \bigg(\int_S h_t\vdaa \,\di|\valpha| - \int_{S'} h'_t\vdm{\valpha'}{|\valpha'|} \,\di|\valpha'|\bigg)
            \tensor \vdbb(t) \,|\vbeta|(\di t)\bigg\|_1 \\
        & \leq \int_T \bigg\|\bigg(\int_S h_t \,\di\valpha - \int_{S'} h'_t \,\di\valpha'\bigg)
            \tensor \vdbb(t)\bigg\|_1 \,|\vbeta|(\di t) \\
        & \leq \int_T \bigg\|\bigg(\int_S h_t \,\di\valpha - \int_{S'} h'_t \,\di\valpha'\bigg)\bigg\|_1 \,|\vbeta|(\di t) 
        \leq 2 m M_2\delta
    \end{align*}
    where the first equality follows
    by \autoref{prp:RNderivoftensor} and \autoref{lem:fubinivec} (Fubini),}
    where the second equality follows by linearity of integration,
    where the second inequality follows by \autoref{prp:aboutvdaa},
    and where the last inequality follows by \eqref{eqn:boundoftvwrt1norm} and \eqref{eqn:tvhtdalphahtpdalphapldelta00}.
    This shows that for each $\vx \in \lsk(\valpha \tensor \vbeta)$ there exists $\vx' \in \lsk(\valpha' \tensor \vbeta)$
    such that $\|\vx - \vx'\|_1 \leq 2 m M_2 \delta$.

    Similarly, for each $\vx' \in \lsk(\valpha' \tensor \vbeta)$ there exists $\vx'' \in \lsk(\valpha' \tensor \vbeta')$
    such that $\|\vx' - \vx''\|_1 \leq 2 n M_1 \delta$. 
    Thus for each $\vx \in \lsk(\valpha \tensor \vbeta)$ there exists $\vx'' \in \lsk(\valpha' \tensor \vbeta')$
    sush that $\|\vx - \vx''\|_1 \leq 2 (n M_1 + m M_2) \delta$, so that
    $$
    \sup_{\vx \in \lsk(\valpha \tensor \vbeta)} d\big(\vx, \lsk(\valpha' \tensor \vbeta')\big) 
    \leq 2 (n M_1 + m M_2) \delta,
    $$
    i.e., $\lsk(\valpha \tensor \vbeta) \subseteq_\veps \lsk(\valpha' \tensor \vbeta')$ for 
    $\veps \geq 2 (n M_1 + m M_2) \delta$.
\end{proof}

\noindent
A dissection of the proof of \autoref{prp:discreteskeletonproductalmostinclusion}
reveals that the condition
$$
\lsk(\valpha), \lsk(\valpha') \subseteq \cube^n(M_1), \quad \lsk(\vbeta), \lsk(\vbeta') \subseteq \cube^m(M_2)
$$
can be relaxed to
$$
\lsk(\valpha') \subseteq \cube^n(M_1), \quad \lsk(\vbeta) \subseteq \cube^m(M_2)
$$
as the proof makes use of, or to
$$
\lsk(\valpha) \subseteq \cube^n(M_1), \quad \lsk(\vbeta') \subseteq \cube^m(M_2)
$$
per a symmetric version of the proof in which $\lsk(\valpha \tensor \vbeta')$
is used as an intermediate point 
between $\lsk(\valpha \tensor \vbeta)$ and $\lsk(\valpha' \tensor \vbeta')$.

\begin{proposition}\label{prp:fubiniwithatoms}
    Let $\valpha$, $\vbeta$ be $\rr^n$-, $\rr^m$-valued measure on $(S, \F)$, $(T, \G)$, respectively.
    Let $A$ be an atom of $|\valpha|$ and $B$ be an atom of $|\vbeta|$.
    Let $f: S \times T \to \rr$ be bounded and measurable.
    Then
    $$
    \int_{S \times B} f \,\di(\valpha \tensor \vbeta) = \int_S f(s, t_0) \,\valpha(\di s) \tensor \vbeta(B)
    $$
    for some $t_0 \in B$,
    $$
    \int_{A \times T} f \,\di(\valpha \tensor \vbeta) = \valpha(A) \tensor \int_T f(s_0, t) \,\vbeta(\di t)
    $$
    for some $s_0 \in A$, and
    $$
    \int_{A \times B} f \,\di(\valpha \tensor \vbeta) = f(s_0, t_0)\big(\valpha(A) \tensor \vbeta(B)\big)
    $$
    for some $s_0 \in A$, $t_0 \in B$.
\end{proposition}
\begin{proof}
    Let $\valpha = (\alpha_1, \dots, \alpha_n)$, $\vbeta = (\beta_1, \dots, \beta_m)$.
    Since $f$ and the Radon-Nikodym derivatives $\lm{\alpha_i}{|\valpha|}$, $\lm{\beta_j}{|\vbeta|}$ are bounded,
    Fubini's theorem (\autoref{lem:fubinivec}) implies that the function
    $h^S_{i,j}: T \to \rr$ defined by 
    $$
    h^S_{i,j}(t) = \int_S f(s, t) \lm{\alpha_i}{|\valpha|}(s) \lm{\beta_j}{|\vbeta|}(t) \,|\valpha|(\di s)
    $$
    for each $(i$, $j) \in [n] \times [m]$ is measurable, 
    so that the function $(\vec{h}^S, \vdbb): T \to (\rr^n \tensor \rr^m) \times \rr^m$
    is measurable and $|\vbeta|$-almost constant on $B$ by (ii) in \autoref{prp:atomsofvecmeasures},
    where $\vec{h}^S: T \to \rr^n \tensor \rr^m$ is defined by
    \begin{align*}
        \vec{h}^S(t)
        & = \int_S f(s, t) (\vdaa(s) \tensor \vdbb(t)) \,|\valpha|(\di s)\\
        & = \int_S f(s, t) \,\valpha(\di s) \tensor \vdbb(t)
    \end{align*}
    for all $t \in T$.
    So there exists
    $t_0 \in B$ such that $(\vec{h}^S, \vdbb) = (\vec{h}^S, \vdbb)(t_0)$ $|\vbeta|$-almost everywhere, and
    \begin{align*}
        & \quad \ \int_{S \times B} f \,\di(\valpha \tensor \vbeta) \\
        & = \int_{S \times B} f(s, t) (\vdaa(s) \tensor \vdbb(t)) \,(|\valpha| \times |\vbeta|)(\di(s, t)) \\
        & = \int_B \int_S f(s, t) (\vdaa(s) \tensor \vdbb(t)) \,|\valpha|(\di s) \,|\vbeta|(\di t) \\
        & = \int_B \vec{h}^S(t) \,|\vbeta|(\di t) \\
        & = \int_B \vec{h}^S(t_0) \,|\vbeta|(\di t) \\
        & = \bigg(\int_S f(s, t_0) \,\valpha(\di s) \tensor \vdbb(t_0)\bigg)|\vbeta|(B) \\
        & = \int_S f(s, t_0) \,\valpha(\di s) \tensor \vbeta(B)
    \end{align*}
    where the second equality follows by Fubini's theorem (\autoref{lem:fubinivec}).
    The other claims follow by similar arguments.
\end{proof}

We separate the approximation argument into nonatomic rounding and
an estimate for atomic factors, then combine them.

\begin{lemma}[Nonatomic rounding]\label{lem:nonatomicrounding}
Let $\valpha$ and $\vbeta$ be nonatomic $\rr^n$- and $\rr^m$-valued
measures on $(S,\F)$ and $(T,\G)$. Write
$A=|\valpha|(S)$ and $B=|\vbeta|(T)$.
For every $\delta>0$ and integer $N\geq1$, there are measures
$\valpha_N,\vbeta_N$ on finite sets, with their power set sigma-algebras,
such that
\begin{align*}
\dhaus(\lsk(\valpha),\lsk(\valpha_N))&\leq(\delta+N^{-1})A,\\
\dhaus(\lsk(\vbeta),\lsk(\vbeta_N))&\leq(\delta+N^{-1})B,\\
\dhaus(\lsk(\valpha\tensor\vbeta),
       \lsk(\valpha_N\tensor\vbeta_N))&\leq(2\delta+N^{-2})AB.
\end{align*}
\end{lemma}
\begin{proof}
If either measure is zero, represent it by a zero mass on a singleton;
the marginal construction below applies to the other measure.
Otherwise partition the $1$-norm unit spheres into finitely many
measurable cells of diameter less than $\delta$. Pull the cells back
under $\vdaa$ and $\vdbb$ to obtain partitions $(S_p)_{p\in[P]}$
and $(T_q)_{q\in[Q]}$, up to null sets, and choose unit representatives
$\vu_p,\vv_q$. Discard cells of zero mass. Put
\[
a_p=|\valpha|(S_p),\qquad b_q=|\vbeta|(T_q).
\]
On $S_p$ and $T_q$, respectively, the densities differ from their
representatives by less than $\delta$ almost everywhere. Define
\[
\valpha_N(\{(p,i)\})=a_p\vu_p/N,\qquad
\vbeta_N(\{(q,j)\})=b_q\vv_q/N,
\quad i,j\in[N].
\]

For a measurable $E\subseteq S$, round
$|\valpha|(E\cap S_p)/a_p$ down to a multiple $k_p/N$ and select
$k_p$ of the $N$ discrete points over $S_p$. Replacing the density
costs at most $\delta a_p$, and rounding costs at most $a_p/N$.
Conversely, a selection of $k_p$ points is matched by a measurable
subset of $S_p$ of mass $k_pa_p/N$, since $|\valpha|$ is nonatomic;
the error is at most $\delta a_p$. Summing proves the first marginal
bound, and the second follows in the same way.

For the product, the $1$-norm satisfies
$\|\vx\tensor\vy\|_1=\|\vx\|_1\|\vy\|_1$, so
\[
\|\vdaa(s)\tensor\vdbb(t)-\vu_p\tensor\vv_q\|_1
\leq2\delta
\]
for almost every $(s,t)\in S_p\times T_q$.
Given a measurable indicator $h$ on $S\times T$, set
\[
r_{pq}=\frac{1}{a_pb_q}
\int_{S_p\times T_q}h\,d(|\valpha|\times|\vbeta|).
\]
Choose an indicator $h_N$ on the discrete product selecting
$\lfloor r_{pq}N^2\rfloor$ of the $N^2$ points over each cell pair.
The density replacement and rounding errors give
\[
\left\|\int h\,d(\valpha\tensor\vbeta)
-\int h_N\,d(\valpha_N\tensor\vbeta_N)\right\|_1
\leq\sum_{p,q}(2\delta+N^{-2})a_pb_q
=(2\delta+N^{-2})AB.
\]
Conversely, if $h_N$ selects a fraction $r_{pq}$ of the points in
that cell pair, nonatomicity gives subsets $E_{pq}\subseteq S_p$
and $F_{pq}\subseteq T_q$ with masses
$\sqrt{r_{pq}}a_p$ and $\sqrt{r_{pq}}b_q$.
The indicator of $\bigcup_{p,q}(E_{pq}\times F_{pq})$ then has the
same cell fractions. Only the density replacement error remains,
bounded by $2\delta AB$. All unions here are finite and measurable.
This proves the product bound.
\end{proof}

\begin{lemma}[Atomic factors]\label{lem:atomicfactorapprox}
Let $\vbeta$ be a purely atomic $\rr^m$-valued measure on $(T,\G)$, with disjoint
atoms $(B_b)_{b\in I}$ of $|\vbeta|$, where $I$ is at most countable.
Let $\vbeta^{\mathrm d}$ be the measure on $I$ with masses
$\vbeta^{\mathrm d}(\{b\})=\vbeta(B_b)$.
Then $\lsk(\vbeta^{\mathrm d})=\lsk(\vbeta)$ and, for every
$\rr^n$-valued measure $\valpha$ on $(S,\F)$,
\[
\lsk(\valpha\tensor\vbeta)
=\lsk(\valpha\tensor\vbeta^{\mathrm d}).
\]
Moreover, if $\valpha'$ is an $\rr^n$-valued measure on
$(S',\F')$ and
$\dhaus(\lsk(\valpha),\lsk(\valpha'))\leq r$, then
\[
\dhaus\bigl(\lsk(\valpha\tensor\vbeta),
\lsk(\valpha'\tensor\vbeta^{\mathrm d})\bigr)
\leq r|\vbeta|(T).
\]
\end{lemma}
\begin{proof}
By \autoref{prp:atomsofvecmeasures}, a measurable subset of an atom
contributes either its entire vector mass or zero. This proves the
marginal equality. By \autoref{prp:fubiniwithatoms}, the product range is
\begin{equation}\label{eq:atomicfactorseries}
\left\{\sum_{b\in I}\vx_b\tensor\vbeta(B_b):
              \vx_b\in\lsk(\valpha)\right\}.
\end{equation}
Indeed, an indicator on $S\times T$ determines one section integral
$\vx_b=\valpha(E_b)$ for each atom. Conversely any such choices are
realized by the indicator of the measurable set
$\bigcup_{b\in I}(E_b\times B_b)$.
The same description holds for $\vbeta^{\mathrm d}$.
All series converge absolutely, since $\lsk(\valpha)$ is bounded and
$\sum_b\|\vbeta(B_b)\|_1=|\vbeta|(T)<\infty$.

For each $\vx_b\in\lsk(\valpha)$ choose
$\vx_b'\in\lsk(\valpha')$ with $\|\vx_b-\vx_b'\|_1\leq r$;
compactness ensures these choices exist. Then
\[
\left\|\sum_b(\vx_b-\vx_b')\tensor\vbeta(B_b)\right\|_1
\leq r\sum_b\|\vbeta(B_b)\|_1=r|\vbeta|(T).
\]
Using \eqref{eq:atomicfactorseries} and making the same choices in the
reverse direction proves the estimate. If $\vbeta=0$, use a zero mass
on a singleton; the assertions are immediate.
\end{proof}

\begin{proposition}\label{prp:lskproddiscreteapprox}
    Let $\valpha$, $\vbeta$ be $\rr^n$-, $\rr^m$-valued measures on $(S, \F)$, $(T, \G)$, respectively.
    Then for any $\veps > 0$ there exist $\rr^n$-, $\rr^m$-valued measures $\valpha'$, $\vbeta'$
    on power set $\sigma$-algebras $\F'$, $\G'$ of countable sets $S'$, $T'$, respectively,
    such that $\dhaus\big(\lsk(\valpha), \lsk(\valpha')\big) < \veps$, $\dhaus\big(\lsk(\vbeta), \lsk(\vbeta')\big) < \veps$,
    $\dhaus\big(\lsk(\valpha \tensor \vbeta), \lsk(\valpha' \tensor \vbeta')\big) < \veps$.
\end{proposition}
\begin{proof}
Decompose the measures into nonatomic and purely atomic parts,
$\valpha=\valpha_c\oplus\valpha_a$ and
$\vbeta=\vbeta_c\oplus\vbeta_a$, as above.
Write $A_c,A_a,B_c,B_a$ for their respective total variations and put
$A=A_c+A_a$, $B=B_c+B_a$.
For $\delta>0$ and $N\geq1$, apply \autoref{lem:nonatomicrounding}
to the nonatomic pair and compress each atomic part as in
\autoref{lem:atomicfactorapprox}. Define
\[
\valpha'=\valpha_{c,N}\oplus\valpha_a^{\mathrm d},\qquad
\vbeta'=\vbeta_{c,N}\oplus\vbeta_a^{\mathrm d}.
\]
The marginal errors are at most $(\delta+N^{-1})A_c$ and
$(\delta+N^{-1})B_c$.

Partition each product space into four blocks according to whether
each factor is nonatomic or atomic. Their ranges add by Minkowski addition,
since measurable sets can be chosen independently in the four blocks.
The bounds supplied by the two lemmas are
\[
\begin{array}{c|c}
\text{block}&\text{Hausdorff error}\\\hline
c\times c&(2\delta+N^{-2})A_cB_c\\
c\times a&(\delta+N^{-1})A_cB_a\\
a\times c&(\delta+N^{-1})A_aB_c\\
a\times a&0.
\end{array}
\]
For the atomic--atomic block apply the equality in
\autoref{lem:atomicfactorapprox} in each factor. For the two mixed
blocks use that lemma with the marginal estimates from
\autoref{lem:nonatomicrounding}.
Hausdorff errors add under Minkowski addition, so the product error
is at most the sum in the table, hence at most
$(2\delta+N^{-1})AB$.
Choose $\delta>0$ small enough and then $N$ large enough to make
this bound and both marginal bounds less than $\veps$.
This also covers zero parts, using zero singleton measures.
\end{proof}

\noindent\textsc{Inclusion and continuity.} The discrete approximation
now gives the following approximate-inclusion estimate, uniform on bounded
families of ranges. The bounds cannot simply be omitted, as shown by
\autoref{prp:unboundedcounterexample} below.
\begin{proposition}\label{prp:almostinclusionpreservationofskeletonprod}
Let $\valpha,\valpha'$ be $\rr^n$-valued and $\vbeta,\vbeta'$ be
$\rr^m$-valued measures, with
\[
\lsk(\valpha),\lsk(\valpha')\subseteq\cube^n(M_1),\qquad
\lsk(\vbeta),\lsk(\vbeta')\subseteq\cube^m(M_2).
\]
If $\lsk(\valpha)\subseteq_\delta\lsk(\valpha')$ and
$\lsk(\vbeta)\subseteq_\delta\lsk(\vbeta')$, where $M_1,M_2,\delta\geq0$, then
\[
\lsk(\valpha\tensor\vbeta)
\subseteq_{2\delta(nM_1+mM_2)}\lsk(\valpha'\tensor\vbeta').
\]
\end{proposition}
\begin{proof}
For $\eta>0$, apply \autoref{prp:lskproddiscreteapprox} separately to
the two pairs. Write the discrete approximants as
$\vlambda,\vomega$ and $\vlambda',\vomega'$. Every marginal and
product range is within Hausdorff distance $\eta$ of its original.
The approximating marginal ranges lie in the cubes with radii
$M_1+\eta,M_2+\eta$, and satisfy the directed inclusions with
error $\delta+2\eta$. By \autoref{prp:discreteskeletonproductalmostinclusion}
and transitivity, the directed distance between the original product
ranges is at most
\[
2\eta+2(\delta+2\eta)\bigl(n(M_1+\eta)+m(M_2+\eta)\bigr).
\]
Let $\eta\downarrow0$ to obtain the claimed bound.
\end{proof}

\begin{theorem}\label{thm:inclusionpreservationLS}
    For any $\rr^n$-valued measures $\valpha$, $\valpha'$ and $\rr^m$-valued measures $\vbeta$, $\vbeta'$,
    $\lsk(\valpha \tensor \vbeta) \subseteq \lsk(\valpha' \tensor \vbeta')$
    if $\lsk(\valpha) \subseteq \lsk(\valpha')$ and $\lsk(\vbeta) \subseteq \lsk(\vbeta')$.
\end{theorem}
\begin{proof}
    Choose finite $M_1,M_2$ bounding the four ranges and apply
    \autoref{prp:almostinclusionpreservationofskeletonprod} with $\delta=0$.
    The target product range is compact by \autoref{lem:codexcompact},
    so zero directed distance implies inclusion.
\end{proof}

\begin{corollary}\label{cor:skeletontensor}
For Lorenz skeletons $K_1=\lsk(\valpha)$ and $K_2=\lsk(\vbeta)$,
\[
K_1\tensor K_2:=\lsk(\valpha\tensor\vbeta)
\]
is independent of the representing measures and is jointly continuous
in the Hausdorff distance.
\end{corollary}
\begin{proof}
Apply \autoref{thm:inclusionpreservationLS} in both directions for
independence. Apply \autoref{prp:almostinclusionpreservationofskeletonprod}
in both directions for continuity, noting that convergent families of
compact sets are bounded.
\end{proof}

One may wonder whether, for fixed source measures $\alpha,\beta$ and
$\veps>0$, one can choose $\delta>0$ independently of the sizes of
the target ranges, so that
\[
\lsk(\alpha)\subseteq_\delta\lsk(\alpha'),\qquad
\lsk(\beta)\subseteq_\delta\lsk(\beta')
\quad\Longrightarrow\quad
\lsk(\alpha\tensor\beta)
\subseteq_\veps\lsk(\alpha'\tensor\beta').
\]
This is false, as shown by the following counterexample, which was
obtained for us by Astra GPT-6.

\begin{proposition}\label{prp:unboundedcounterexample}
There are fixed scalar measures $\alpha,\beta$ and sequences of scalar
measures $\alpha_L,\beta_L$ such that
\[
\lsk(\alpha)\subseteq_{1/L}\lsk(\alpha_L),\qquad
\lsk(\beta)\subseteq_{1/L}\lsk(\beta_L),
\]
but
\[
\lim_{L\to\infty}\operatorname{dist}
\bigl(3,\lsk(\alpha_L\tensor\beta_L)\bigr)=1,
\qquad 3\in\lsk(\alpha\tensor\beta).
\]
In particular, the approximate-inclusion conclusion cannot hold with
a choice of $\delta>0$ depending only on $\alpha,\beta,\veps$ and
valid for all target measures $\alpha',\beta'$.
\end{proposition}

\begin{proof}
Let $\alpha=\beta$ be the measure on a two-point space assigning mass
$1$ to each point. Then
\[
\lsk(\alpha)=\lsk(\beta)=\{0,1,2\},\qquad
\lsk(\alpha\tensor\beta)=\{0,1,2,3,4\}.
\]
For integers $L\geq3$, let $\alpha_L=\beta_L$ be the signed measure
on a three-point space with masses
\[
x=L,\qquad y=1-L,\qquad z=2-L+L^{-1}.
\]
Its range contains $0$, $x+y=1$, and $x+z=2+L^{-1}$, proving
the two approximate inclusions.

The product measure has nine atoms, with masses $x^2,y^2,z^2$ and
two copies each of $xy,xz,yz$. Every point of its range therefore has
the form
\[
P_L=ax^2+by^2+cz^2+dxy+exz+fyz,
\quad a,b,c\in\{0,1\},\quad d,e,f\in\{0,1,2\}.
\]
Expanding gives
\begin{align*}
P_L={}&(a+b+c-d-e+f)L^2
       +(-2b-4c+d+2e-3f)L\\
     &+(b+2c+e+f)+(4c+f)L^{-1}+cL^{-2}.
\end{align*}
For a fixed coefficient tuple, this expression diverges in absolute
value unless both coefficients of positive powers of $L$ vanish.
Their vanishing is equivalent to
\[
d=2a-2c-f,\qquad e=b+3c+2f-a.
\]
The allowed coefficient tuples are:
\[
\begin{array}{c|c|c|c|c|c|l}
a&b&c&d&e&f&P_L\\ \hline
0&0&0&0&0&0&0\\
0&1&0&0&1&0&2\\
1&0&0&1&1&1&2+L^{-1}\\
1&0&1&0&2&0&4+4L^{-1}+L^{-2}\\
1&1&0&2&0&0&1\\
1&1&0&1&2&1&4+L^{-1}
\end{array}
\]
There are only finitely many coefficient tuples. Thus all other
values eventually leave any fixed bounded interval, while the values
in the table converge to $0,1,2$, or $4$. Their minimum distance to
$3$ tends to $1$, as claimed. Taking, for example, $\veps=1/2$
disproves the asserted uniform choice of $\delta$ for unrestricted
target ranges.
\end{proof}

\section{Support Functions of Lorenz Hulls}\label{sec:lorenzhulls}

\begin{definition}\label{def:suppfunc}
    Let $C \subseteq \rr^n$.
    The \emph{support function of $C$},
    written $\dual{C}{\cdot}$, is the function
    from $\rr^n$
    to $\Rbar$ defined by
    $\dual{C}{\vxs} \coloneqq \sup \{\dpr{\vxs}{\vx} \,:\, \vx \in C\}$.
\end{definition}

\begin{proposition}\label{prp:suppconju}
    If $C \subseteq \rr^n$ convex and closed then $D \subseteq C$
    if and only if $\dual{D}{\vxs} \leq \dual{C}{\vxs}$ for all $\vxs \in \rr^n$.
\end{proposition}

\noindent
For a measurable function $\vf: S \to \rr^n$ on $(S, \F)$, let
$$
\domc{\dpr{\vxs}{\vf}}{\star}{0} \coloneqq \{s \in S \,:\, \dpr{\vxs}{\vf(s)} \star 0\}
$$
for $\vxs \in \rr^n$, $\star \in \{>, \geq, =, \leq, <\}$.
We note that $\domc{\dpr{\vxs}{\vf}}{\star}{0} \in \F$ for any $\vxs$ and $\star$
since $\{\vz \in \rr^n \,:\, \dpr{\vxs}{\vz} \star 0\}$
is either a hyperplane or a (closed or open) half-space in $\rr^n$
and since $\vf$ is measurable.

\begin{proposition}\label{prp:threecases}
    Let $\valpha$ be an $\rr^n$-valued measure on $(S, \F)$,
    let $\mu$ be a finite positive measure dominating $\valpha$ on $(S, \F)$,
    and let $\vxs \in \rr^n$, $A \in \F$. If $\mu(A) > 0$, then
    $\dpr{\vxs}{\valpha(A)} > 0$, $\dpr{\vxs}{\valpha(A)} = 0$, $\dpr{\vxs}{\valpha(A)} < 0$
    if $A \subseteq \Diamu{>}$, $A \subseteq \Diamu{=}$, $A \subseteq \Diamu{<}$, respectively.
    If $\mu(A) = 0$ then $\dpr{\vxs}{\valpha(A)} = 0$.
\end{proposition}
\begin{proof}
    If $\mu(A) = 0$ then $\valpha(A) = 0$ since $\mu$ dominates $\valpha$, and then $\dpr{\vxs}{\valpha(A)} = 0$.
    For when $\mu(A) > 0$, we will only discuss the situation where $A \subseteq \Diamu{>}$ and omit the proof for the other situations.
    Let
    $$
    A_\veps = \{s \in A \,:\, \dpr{\vxs}{\vdamu(s)} \geq \veps\}
    $$
    for $\veps > 0$. Then $\mu(A_\veps) = \delta > 0$ for some $\veps > 0$
    since $A = \bigcup_{k = 1}^\infty A_{1/k}$. Thus
    \begin{align*}
        \dpr{\vxs}{\valpha(A)}
        & = \dpr{\vxs}{\valpha(A_\veps)} + \dpr{\vxs}{\valpha(A \backslash A_\veps)} \\
        & = \int_{A_\veps} \dpr{\vxs}{\vdamu(s)} \,\mu(\di s)
            + \int_{A \backslash A_\veps} \dpr{\vxs}{\vdamu(s)} \,\mu(\di s) \\
        & \geq \veps\delta + 0 > 0,
    \end{align*}
    which is the desired result.
\end{proof}

\begin{proposition}\label{prp:reachofld}
    Let $\valpha$ be an $\rr^n$-valued measure on $(S, \F)$
    and let $\mu$ be a finite positive measure dominating $\valpha$ on $(S, \F)$.
    Then 
    $$
    \dual{\ld(\valpha)}{\vxs} = \dpr{\vxs}{\valpha(\Diamu{>})} = \dpr{\vxs}{\valpha(\Diamu{\geq})}
    $$
    for all $\vxs \in \rr^n$.
\end{proposition}
\begin{proof}
    Fix any $\vxs \in \rr^n$.
    \autoref{prp:threecases} implies that $\dpr{\vxs}{\valpha(\Diamu{>})} \geq \dpr{\vxs}{\valpha(B)}$
    for all $B \in \F$.
    Thus 
    $$
    \dual{\lsk(\valpha)}{\vxs} = \dpr{\vxs}{\valpha(\Diamu{>})}
    $$
    since $\valpha(\Diamu{>}) \in \lsk(\valpha)$.
    But $\dual{\lsk(\valpha)}{\vxs} = \dual{\ld(\valpha)}{\vxs}$ 
    since $\ld(\valpha)$ is the convex hull of $\lsk(\valpha)$.
    The same argument holds replacing $\Diamu{>}$ by $\Diamu{\geq}$.
\end{proof}

\section{Inclusion Preservation of Lorenz Hulls}\label{sec:lorenzprod}
\noindent In this section all distances and almost inclusions use
$p=\|\cdot\|_2$, including the standard Euclidean norm on
$\rr^n\tensor\rr^m\cong\rr^{nm}$. Write $\sfU_2^k$ for the closed
Euclidean unit ball in $\rr^k$.

For $\vy\in\rr^m$, let $\trany{\vy}:\rr^n\tensor\rr^m\to\rr^n$
be the contraction characterized by
\begin{equation}\label{eqn:tensorinntrans}
\dpr{\trany{\vy}(\vu)}{\vx}=\dpr{\vu}{\vx\tensor\vy}.
\end{equation}
Cauchy--Schwarz, followed by the supremum over $\|\vx\|_2=1$, gives
\begin{equation}\label{eq:contractionbound}
\|\trany{\vy}(\vu)\|_2\leq\|\vy\|_2\|\vu\|_2.
\end{equation}
We use the following support-function identity, the counterpart for
Lorenz hulls of \cite[Eq.~(3.2)]{Breiding2022TheZonoid}.

\begin{lemma}\label{lem:hullproductsupport}
Let $\valpha,\vbeta$ be $\rr^n$- and $\rr^m$-valued measures on
$(S,\F)$ and $(T,\G)$, and let $\nu$ be a finite positive measure
 dominating $\vbeta$. Put $\vg=d\vbeta/d\nu$. Then
\begin{equation}\label{eq:hullproductsupport}
h_{\ld(\valpha\tensor\vbeta)}(\vu)
=\int_T h_{\ld(\valpha)}\bigl(\trany{\vg(t)}(\vu)\bigr)\,d\nu(t).
\end{equation}
\end{lemma}
\begin{proof}
Choose a finite positive measure $\mu$ dominating $\valpha$ and put
$\vf=d\valpha/d\mu$. By \autoref{prp:reachofld},
$h_{\ld(\valpha)}(\vx)=\int_S\max\{0,\dpr{\vx}{\vf(s)}\}\,d\mu(s)$.
The product density is $\vf(s)\tensor\vg(t)$ by
\autoref{prp:RNderivoftensor}. Substitution, \eqref{eqn:tensorinntrans},
and Fubini's theorem (\autoref{lem:fubinivec}) give \eqref{eq:hullproductsupport}; the integrand is bounded
by the integrable function $\|\vu\|_2\|\vf(s)\|_2\|\vg(t)\|_2$.
\end{proof}

\begin{theorem}\label{prp:containpresprodnd}
    Let $\rr^n$-valued measures $\valpha$, $\valpha'$, $\rr^m$-valued measures $\vbeta$, $\vbeta'$ 
    be such that
    $\ld(\valpha) \subseteq \ld(\valpha')$, $\ld(\vbeta) \subseteq \ld(\vbeta')$.
    Then $\ld(\valpha \tensor \vbeta) \subseteq \ld(\valpha' \tensor \vbeta')$.
\end{theorem}
\begin{proof}
The hypothesis $\ld(\valpha)\subseteq\ld(\valpha')$ gives
$h_{\ld(\valpha)}\leq h_{\ld(\valpha')}$.
Integrating this inequality in \eqref{eq:hullproductsupport} proves
$\ld(\valpha\tensor\vbeta)\subseteq\ld(\valpha'\tensor\vbeta)$.
Apply the same argument in the second factor.
\end{proof}

\section{Products of Lorenz Hulls}\label{sec:geometricproducts}
For Lorenz hulls $U=\ld(\valpha)$ and $V=\ld(\vbeta)$, define
\[
U\tensor V:=\ld(\valpha\tensor\vbeta).
\]
The definition is independent of the measures by
\autoref{prp:containpresprodnd}, applied in both directions.
The product preserves inclusion in both arguments.

Under the canonical identifications
$V\tensor W\cong W\tensor V$ and
$(V\tensor W)\tensor X\cong V\tensor(W\tensor X)$,
the tensor product is commutative and associative. Indeed, the
corresponding measure products agree on measurable rectangles after
permuting or regrouping the factors, and hence agree by
\autoref{lem:rmeasureagree}. This argument also applies to skeletons.

\section{Hausdorff Continuity of Lorenz Hull Products}\label{sec:hullcontinuity}
All distances in this section use the Euclidean norm, as in
\autoref{sec:lorenzprod}.
\begin{proposition}\label{prp:intof2normofRDderiv}
    Let $\valpha$ be an $\rr^n$-valued measure on $(S,\F)$ 
    and let $\mu$ be a finite positive measure dominating $\valpha$.
    If $\ld(\valpha) \subseteq M \sfU_2^n$ for some $M \geq 0$,
    where $\sfU^n_2$ is the unit ball in $\rr^n$ with respect to the $2$-norm, then
    $$
    \int_S \bigg\|\vdamu\bigg\|_2 \,\di\mu \leq M \tau_n,
    $$
    where $\tau_n = \sqrt{2\pi} \bbE \|\vv\|_2$ for $\vv$ a standard Gaussian random vector in $\rr^n$. 
\end{proposition}
\begin{proof}
Put $\vf=d\valpha/d\mu$, $H=\ld(\valpha)$, and let $\vv$ be a
standard Gaussian vector. The identity
$\|\vx\|_2=\sqrt{\pi/2}\,\bbE|\dpr{\vx}{\vv}|$ and Tonelli give
\[
\int_S\|\vf\|_2\,d\mu
=\sqrt{\frac\pi2}\,\bbE\bigl(h_H(\vv)+h_H(-\vv)\bigr)
\leq\sqrt{2\pi}\,M\bbE\|\vv\|_2=M\tau_n.
\]
Here the equality uses \autoref{prp:reachofld}, and the inequality
uses $H\subseteq M\sfU_2^n$.
\end{proof}

\begin{proposition}\label{prp:vepsalmincludandsuppfunc}
    If $A$, $B \subseteq \rr^n$ are nonempty, convex and compact and $\veps\geq0$, then $A \subseteq_\veps B$ if and only if
    $$
    h_A(\vxs) \leq h_B(\vxs) + \veps\|\vxs\|_2
    $$
    for all $\vxs \in \rr^n$.
\end{proposition}
\begin{proof}
    Since $B$ is compact, \autoref{prp:vepsincludiffforxexistsy} gives 
    $$
    A \subseteq_\veps B \iff A \subseteq B + \veps \sfU_2^n,
    $$
    which, since $A$, $B$ are convex and compact, is equivalent to
    $$
    h_A(\vxs) \leq h_{B + \veps \sfU_2^n}(\vxs)
    $$
    while
    $$
    h_{B + \veps \sfU_2^n}(\vxs) = h_B(\vxs) + \veps h_{\sfU_2^n}(\vxs) = h_B(\vxs) + \veps\|\vxs\|_2
    $$
    for all $\vxs \in \rr^n$.
\end{proof}

\begin{proposition}\label{prp:almostinclusionpreservationofLHprod}
    Let $\valpha$, $\valpha'$ be $\rr^n$-valued measures
    and $\vbeta$, $\vbeta'$ be $\rr^m$-valued measures such that
    $$
    \ld(\valpha), \ld(\valpha') \subseteq M \sfU_2^n, \quad \ld(\vbeta), \ld(\vbeta') \subseteq M \sfU_2^m
    $$
    for some $M \geq 0$. 
    If $\ld(\valpha) \subseteq_\delta \ld(\valpha')$, $\ld(\vbeta) \subseteq_\delta \ld(\vbeta')$
    for $\delta \geq 0$, then
    $$
    \ld(\valpha \tensor \vbeta) \subseteq_\veps \ld(\valpha' \tensor \vbeta')
    $$
    for $\veps \geq M(\tau_n + \tau_m) \delta$.
\end{proposition}
\begin{proof}
We use the support-function argument of
\cite[Lemma~3.9]{Breiding2022TheZonoid}, retaining a one-sided inequality.
Put $H=\ld(\valpha)$, $H'=\ld(\valpha')$, and choose a finite positive
measure $\nu$ dominating $\vbeta$. Write $\vg=d\vbeta/d\nu$ and
$C_\beta=\int_T\|\vg\|_2\,d\nu$.
The assumption $H\subseteq_\delta H'$ means
$h_H(\vx)-h_{H'}(\vx)\leq\delta\|\vx\|_2$ by
\autoref{prp:vepsalmincludandsuppfunc}. Consequently,
\begin{align*}
h_{\ld(\valpha\tensor\vbeta)}(\vu)
-h_{\ld(\valpha'\tensor\vbeta)}(\vu)
&=\int_T(h_H-h_{H'})\bigl(\trany{\vg(t)}(\vu)\bigr)\,d\nu(t)\\
&\leq\delta\int_T\|\trany{\vg(t)}(\vu)\|_2\,d\nu(t)\\
&\leq\delta C_\beta\|\vu\|_2,
\end{align*}
by \eqref{eq:hullproductsupport} and \eqref{eq:contractionbound}.
Thus the first change of factor costs at most $\delta C_\beta$.
Changing the second factor similarly costs at most $\delta C_{\alpha'}$,
where $C_{\alpha'}=\int\|d\valpha'/d\mu'\|_2\,d\mu'$ for a finite
positive measure $\mu'$ dominating $\valpha'$.
Adding the two directed errors gives the intermediate bound
\[
\ld(\valpha\tensor\vbeta)
\subseteq_{\delta(C_\beta+C_{\alpha'})}
\ld(\valpha'\tensor\vbeta').
\]
Finally, \autoref{prp:intof2normofRDderiv} gives
$C_\beta\leq M\tau_m$ and $C_{\alpha'}\leq M\tau_n$.
\end{proof}

For any four fixed hulls, a finite common radius bound $M$ exists;
setting $\delta=0$ in the preceding proposition therefore recovers
exact inclusion preservation (\autoref{prp:containpresprodnd}), since
the target hull is closed. Applying the estimate in both directions
gives Hausdorff continuity:

\begin{corollary}\label{cor:codexcontinuity}
On the bounded families of \autoref{prp:almostinclusionpreservationofLHprod},
\[
\begin{split}
\dhaus(U\tensor V,U'\tensor V')
&\leq M(\tau_n+\tau_m)\\
&\quad\cdot\max\{\dhaus(U,U'),\dhaus(V,V')\}.
\end{split}
\]
In particular the Lorenz-hull tensor product is jointly continuous.
\end{corollary}
\begin{proof}
Apply the proposition in both directions. Convergent sequences of
compact sets have uniformly bounded radii, so the bounded estimate
implies joint continuity without a global bound on all hulls.
\end{proof}

\section{Sums and Distributivity}\label{sec:sums}
The distributivity follows by an analogous property of products of
disjoint sums of measures. We use the disjoint sum defined in
\autoref{sec:prodofskeletons}. Here $A+B=\{\vx+\vy:\vx\in A,\vy\in B\}$
is the Minkowski sum. We use the elementary identity
\[
\convhull(A+B)=\convhull(A)+\convhull(B).
\]
One inclusion follows from convexity of the right-hand side; for the
other, expand a sum of two convex combinations using the products of
their coefficients.
\begin{proposition}\label{prp:ldsum}
    Let $\valpha$, $\vbeta$ be $n$-dimensional real vector measures on $(S, \F)$, $(T, \G)$, respectively, with $S$, $T$ disjoint.
    Then $\ld(\valpha \oplus \vbeta) = \ld(\valpha) + \ld(\vbeta)$.
\end{proposition}
\begin{proof}
    It is easy to check that $\lsk(\valpha \oplus \vbeta) = \lsk(\valpha) + \lsk(\vbeta)$.
    Taking convex hulls and applying the preceding identity proves the claim.
\end{proof}

\noindent
The product of real vector measures is distributive over the disjoint sum operation ``$\oplus$'':

\begin{proposition}\label{prp:measuredistrb}
    Let $\valpha$, $\vbeta$ be $n$-dimensional real vector measures on $(S, \F)$, $(T, \G)$, respectively, with $S$, $T$ disjoint.
    Let $\vtau$ be an $m$-dimensional real vector measure on $(\Omega, \H)$.
    Then $(\valpha \oplus \vbeta) \tensor \vtau = (\valpha \tensor \vtau) \oplus (\vbeta \tensor \vtau)$.
\end{proposition}
\begin{proof}
    It is easy to check that the two measures in question have the same domain,
    i.e., that the union-sum of the $\sigma$-algebras of $(S, \F) \times (\Omega, \H)$ and $(T,\G) \times (\Omega, \H)$
    coincides with the $\sigma$-algebra of $(S \cup T, \F \oplus \G) \times (\Omega, \H)$.
    Moreover, it is also easy to check that the two measures agree on sets of the form $A \times B$
    where $A \in \F \oplus \G$, $B \in \H$.
    The statement thus follows by \autoref{lem:rmeasureagree}.
\end{proof}

\begin{proposition}\label{prp:lorenzdistrblaw}
For Lorenz hulls $H_1,H_2\subseteq\rr^n$ and $H_3\subseteq\rr^m$,
\[
(H_1+H_2)\tensor H_3=(H_1\tensor H_3)+(H_2\tensor H_3).
\]
The analogous identity holds in the second argument and for Lorenz skeletons.
\end{proposition}
\begin{proof}
Choose representing measures on disjoint spaces. The identity follows
from \autoref{prp:ldsum} and \autoref{prp:measuredistrb}.
For skeletons use
$\lsk(\valpha\oplus\vbeta)=\lsk(\valpha)+\lsk(\vbeta)$.
\end{proof}

\section{Multilinear Maps and the Coordinate-wise Product}\label{sec:multilinear}
For a linear map $L:V\to W$ and a $V$-valued measure $\valpha$, put
$(L\valpha)(A)=L(\valpha(A))$. Then
\begin{equation}\label{eq:linearimages}
\lsk(L\valpha)=L(\lsk(\valpha)),\qquad
\ld(L\valpha)=L(\ld(\valpha)).
\end{equation}
The first identity follows from the definition of range, and the second
from preservation of convex combinations by a linear map.

Let $M:V_1\times\cdots\times V_k\to W$ be multilinear.
There is a linear map $\widetilde M:V_1\tensor\cdots\tensor V_k\to W$
with $\widetilde M(v_1\tensor\cdots\tensor v_k)=M(v_1,\ldots,v_k)$.
For vector measures $\valpha_i$ on $(S_i,\F_i)$ define
\[
M(\valpha_1,\ldots,\valpha_k)
:=\widetilde M(\valpha_1\tensor\cdots\tensor\valpha_k).
\]
Equivalently, for finite positive measures $\mu_i$ dominating $\valpha_i$,
\begin{equation}\label{eq:multilineardensity}
\begin{split}
M(\valpha_1,\ldots,\valpha_k)(E)
={}&\int_E M\left(\frac{\di\valpha_1}{\di\mu_1}(s_1),\ldots,
\frac{\di\valpha_k}{\di\mu_k}(s_k)\right)\\
&\hspace{15mm}(\mu_1\times\cdots\times\mu_k)(\di(s_1,\ldots,s_k)).
\end{split}
\end{equation}
The integral is finite because $\|M(v_1,\ldots,v_k)\|\leq
C\prod_i\|v_i\|$ in finite dimensions. Its independence from the
dominating measures follows from \autoref{prp:RNderivoftensor} and
linearity of integration.

\begin{proposition}\label{prp:multilinearhulls}
The hull of $M(\valpha_1,\ldots,\valpha_k)$ depends only on the
individual hulls, and preserves inclusion in each of them. The same
statements hold with skeletons in place of hulls.
\end{proposition}
\begin{proof}
Apply \eqref{eq:linearimages} to the tensor product and use
\autoref{prp:containpresprodnd} or \autoref{thm:inclusionpreservationLS}
successively in each argument.
\end{proof}

\paragraph{Complex vector measures.}
A complex vector measure can be viewed as a real vector measure by
separating the real and imaginary parts of each coordinate. Complex
multiplication is the real-bilinear map
\[
((a,b),(c,d))\longmapsto(ac-bd,ad+bc).
\]
Thus \autoref{prp:multilinearhulls} applies to coordinate-wise complex
multiplication and gives well-defined products of complex measure
ranges and their real convex hulls, preserving inclusion and Hausdorff
continuity. The same reduction applies to any complex-multilinear map,
including the complex tensor product, as a linear image of the real
tensor product.

For $\vx,\vy\in\rr^n$, let $\cwp(\vx,\vy)=(x_1y_1,\ldots,x_ny_n)$.
For $\valpha=(\alpha_1,\ldots,\alpha_n)$ and
$\vbeta=(\beta_1,\ldots,\beta_n)$ define
\[
\valpha\times\vbeta
=(\alpha_1\times\beta_1,\ldots,\alpha_n\times\beta_n).
\]
This is $M(\valpha,\vbeta)$ for $M=\cwp$, equivalently the linear
image of $\valpha\tensor\vbeta$ selecting its diagonal coordinates.
We refer to the induced product as the Lorenz product, writing
$H_1H_2=\ld(\valpha\times\vbeta)$ when $H_1=\ld(\valpha)$,
$H_2=\ld(\vbeta)$. It is well-defined by \autoref{prp:multilinearhulls}.
The same definition applies to skeletons.

The Lorenz product is commutative and associative, by the corresponding
properties of the scalar product measures. It preserves inclusion and
is distributive over Minkowski sums, by \autoref{prp:multilinearhulls}
and the disjoint-sum argument of \autoref{sec:sums}.
The coordinate projection is continuous, so the Hausdorff continuity
of the tensor product also gives continuity of this product.

\begin{lemma}\label{lem:pointproductinclusion}
For every $\mathbf{h}\in\lsk(\valpha)$ and $\mathbf{k}\in\lsk(\vbeta)$,
$\cwp(\mathbf{h},\mathbf{k})\in\lsk(\valpha\times\vbeta)$.
Likewise, for every $\mathbf{h}\in\ld(\valpha)$ and $\mathbf{k}\in\ld(\vbeta)$,
$\cwp(\mathbf{h},\mathbf{k})\in\ld(\valpha\times\vbeta)$.
The analogous statements hold for tensor products.
\end{lemma}
\begin{proof}
For skeletons, choose measurable $E,F$ with
$\mathbf{h}=\valpha(E)$ and $\mathbf{k}=\vbeta(F)$. Then
\[
\cwp(\mathbf{h},\mathbf{k})=(\valpha\times\vbeta)(E\times F).
\]
For hulls, write $\mathbf{h}=\sum_i a_i\valpha(E_i)$ and
$\mathbf{k}=\sum_j b_j\vbeta(F_j)$ as finite convex combinations. Bilinearity gives
\[
\cwp(\mathbf{h},\mathbf{k})=\sum_{i,j}a_i b_j
(\valpha\times\vbeta)(E_i\times F_j),
\]
a convex combination of points in the product range. The same proof
applies with $\tensor$ in place of the coordinate-wise product.
\end{proof}

Its multiplicative identity on hulls is
$\convhull(\{\mathbf0,\mathbf1\})$, where $\mathbf1=(1,\ldots,1)$:
represent this segment by a vector measure of mass $\mathbf1$ on a
singleton. The corresponding identity on skeletons is $\{\mathbf0,\mathbf1\}$.

\begin{proposition}\label{prp:invertiblehulls}
A real Lorenz hull is invertible if and only if it is a segment
$[\mathbf0,\vx]$ with every $x_i\ne0$. Its inverse is
$[\mathbf0,(1/x_1,\ldots,1/x_n)]$.
\end{proposition}
\begin{proof}
The displayed segments are inverse to each other by representing
them with singleton measures.

Conversely, suppose $HK=[\mathbf0,\mathbf1]$.
Neither $H$ nor $K$ can lie in a coordinate hyperplane: the corresponding
coordinate measure would be zero, so that coordinate would vanish
throughout $HK$.
There is consequently a $\mathbf{k}\in K$ with every $k_i\ne0$.
Indeed, one can successively make each coordinate nonzero by taking
a convex combination with a point of $K$ where that coordinate is
nonzero; only finitely many choices of the mixing parameter can make
any previously nonzero coordinate vanish.

By \autoref{lem:pointproductinclusion}, the invertible linear map
$D_{\mathbf{k}}(\mathbf{h})=\cwp(\mathbf{h},\mathbf{k})$ sends $H$ into
$[\mathbf0,\mathbf1]$. Since $H$ is compact and convex and contains
$\mathbf0$, its image is $[\mathbf0,t\mathbf1]$ for some $t\in[0,1]$.
The case $t=0$ would give $H=\{\mathbf0\}$ and hence
$HK=\{\mathbf0\}$, a contradiction. Thus
\[
H=[\mathbf0,\vx],\qquad
\vx=(t/k_1,\ldots,t/k_n),
\]
and all coordinates of $\vx$ are nonzero.
\end{proof}

\section{Zonoids Defined by Support Functions}\label{sec:zonoids}

A Lorenz skeleton consists of the integrals $\int_S f\,\di\valpha$
with measurable $f:S\to\{0,1\}$. Allowing $f$ to take values in
$[0,1]$ gives the hull:

\begin{proposition}\label{prp:ldavhull}
    Let $\valpha$ be an $n$-dimensional finite signed measure on $(S, \F)$. Then 
    $$
    \ld(\valpha) = \Big\{\int_S f \,\di\valpha \,:\, f: S \to [0, 1] \text{ is measurable}\Big\}.
    $$
\end{proposition}
The proof is given in Appendix~\ref{sec:equivalences}. That appendix
also shows that every Lorenz hull is the range of a nonatomic vector
measure (\autoref{prp:ldiszonoid}).

\begin{definition}\label{def:muintegrablefunc}
    Let $\mu$ be a finite positive measure on $(S, \F)$. 
    A measurable function $\vf: S \to \rr^n$ is called \emph{absolutely $\mu$-integrable} if $
    \int_S \|\vf\|_1 \,\di\mu < \infty.
    $
\end{definition}

\noindent
Let $[\va, \vb]$ denote the closed segment in $\rr^n$ with endpoints $\va$ and $\vb \in \rr^n$.
For any absolutely $\mu$-integrable function $\vf: S \to \rr^n$, let 
$$
\sfZ(\vf, \mu) 
$$
be the compact convex subset of $\rr^n$ whose support function is calculated by
$$
h_{\sfZ(\vf, \mu)}(\vxs) = \int_S h_{[\nvec{0}, \vf]}(\vxs) \,\di\mu
$$
for all $\vxs \in \rr^n$. 
Indeed, define $\valpha_f(A)=\int_A\vf\,\di\mu$.
Coordinatewise dominated convergence makes this a vector measure, and
\autoref{prp:reachofld} shows that the displayed function is precisely
$h_{\ld(\valpha_f)}$. This also proves $\sfZ(\vf,\mu)=\ld(\valpha_f)$.

Conversely, every Lorenz hull has such a representation:
\begin{proposition}\label{prp:equivofsfKandLH}
Let $\valpha$ be an $\rr^n$-valued measure and $\mu$ a finite positive
measure dominating it. Then $\sfZ(\vdamu,\mu)=\ld(\valpha)$.
\end{proposition}
\begin{proof}
Apply the preceding construction to $\vf=d\valpha/d\mu$;
then $\valpha_f=\valpha$.
\end{proof}

In these representations the tensor product takes the form
\begin{equation}\label{eq:zonoidproductrepresentation}
\sfZ(\vf,\mu)\tensor\sfZ(\vg,\nu)
=\sfZ(\vf\tensor\vg,\mu\times\nu).
\end{equation}
Indeed, the measures $\valpha_f$ and $\vbeta_g$ have product density
$\vf\tensor\vg$ relative to $\mu\times\nu$ by
\autoref{prp:RNderivoftensor}. The identity therefore follows from
\autoref{prp:equivofsfKandLH}; its left side already depends only on
the two hulls.

For an absolutely $\mu$-integrable $\vf$, the centered convention is
\[
h_{\sfK(\vf,\mu)}(\vxs)
=\frac12\int_S|\dpr{\vxs}{\vf}|\,d\mu.
\]
This defines a compact convex set because
\[
\sfK(\vf,\mu)=\tfrac12\sfZ(\vf,\mu)+\tfrac12\sfZ(-\vf,\mu).
\]
The equality follows from $|x|=\max(0,x)+\max(0,-x)$ and
additivity of support functions under Minkowski sums. In particular,
$\sfZ(-\vf,\mu)=-\sfZ(\vf,\mu)$, so $\sfK$ is the centered
symmetrization of $\sfZ$.

For centered zonoids represented as $U=\sfK(\vf,\mu)$ and
$V=\sfK(\vg,\nu)$, the centered construction is
$U\ctensor V=\sfK(\vf\tensor\vg,\mu\times\nu)$.
For probability measures this is the centered convention of
\cite{Breiding2022TheZonoid}; finite nonzero measures reduce to
probability measures by rescaling the vector functions by their total masses.
\begin{proposition}\label{prp:equivoftwoproducts2}
For centered zonoids $U,V$, the centered construction is independent
of the representations and $U\ctensor V=U\tensor V$.
\end{proposition}
\begin{proof}
Let $\rho$ give mass $1/2$ to each of $-1,1$. On
$S\times\{-1,1\}$ define the measure $\widehat\alpha$ with density
$\widehat f(s,e)=e\vf(s)$ relative to $\mu\times\rho$.
Then
\[
h_{\ld(\widehat\alpha)}(x)
=\frac12\int_S\bigl(\max(0,x^T\vf)+\max(0,-x^T\vf)\bigr)\,\di\mu
=h_U(x).
\]
Similarly construct $\widehat\beta$ with density $d\vg(t)$,
where $d$ is an independent sign of law $\rho$.
The product density is $ed(\vf(s)\tensor\vg(t))$ and $ed$ again
has law $\rho$. Consequently
\[
\ld(\widehat\alpha\tensor\widehat\beta)
=\sfK(\vf\tensor\vg,\mu\times\nu).
\]
The left side is $U\tensor V$. This proves both equality and
independence, including zero measures. The usual representation of
centered zonoids by $\sfK$ is recalled in \cite{Breiding2022TheZonoid}.
\end{proof}

Let $\mu_i$ be probability measures dominating $\valpha_i$ on
$(S_i,\F_i)$, and put
\[
X_i=\frac{\di\valpha_i}{\di\mu_i}.
\]
Then $\sfZ(X_i,\mu_i)=\ld(\valpha_i)$, and on the product probability space,
\[
\sfZ(M(X_1,\ldots,X_k),\mu_1\times\cdots\times\mu_k)
=\ld(M(\valpha_1,\ldots,\valpha_k)).
\]
This follows from \eqref{eq:multilineardensity} and
\autoref{prp:equivofsfKandLH}. For centered zonoids use $\sfK$ instead:
\autoref{prp:equivoftwoproducts2}, iteration, and a linear image give
the centered multilinear construction of \cite{Breiding2022TheZonoid}.

\appendix
\section{Integral and Nonatomic Representations}\label{sec:equivalences}

We prove the integral characterization stated in \autoref{prp:ldavhull}
and the existence of a nonatomic representation. We first record a
consequence of \autoref{prp:threecases}.

\begin{proposition}\label{prp:xstarDmuintfcomparingzerosigned}
    Let $\valpha$ be an $n$-dimensional finite signed measure on $(S, \F)$,
    let $\mu$ be a $\sigma$-finite positive measure dominating $\valpha$ on $(S, \F)$,
    let $f: S \to [0, \infty)$ be a measurable function such that $\int_S f \,\di|\valpha| < \infty$, and let $A \in \F$, $\vxs\in\rr^n$.
    Then $\dpr{\vxs}{\int_A  f \,\di\valpha} \geq 0$, $\dpr{\vxs}{\int_A  f \,\di\valpha} \leq 0$
    if $A \subseteq \Diamu{\geq}$, $A \subseteq \Diamu{\leq}$, respectively.
\end{proposition}
\begin{proof}
    Use the fact that
    $$
    f(s) \cdot \dpr{\vxs}{\vdamu(s)} \geq 0, \quad f(s) \cdot \dpr{\vxs}{\vdamu(s)} \leq 0
    $$
    for $s \in \Diamu{\geq}$, $s \in \Diamu{\leq}$, respectively.
    The rest of the proof follows by definition of $\int_A f \,\di\valpha$ and
    by linearity of the ordinary Lebesgue integral.
\end{proof}

\begin{proof}[Proof of \autoref{prp:ldavhull}]
    Let 
    $$
    V = \Big\{\int_S f \,\di\valpha \,:\, f: S \to [0, 1] \text{ is measurable}\Big\}
    $$
    for short.
    It is easy to see that $V$ is convex and that $\lsk(\valpha) \subseteq V$, so that $\ld(\valpha) \subseteq V$.
    For the reverse containment, it suffices to show that
    $$
    \dual{\ld(\valpha)}{\vxs} \geq \dual{V}{\vxs}
    $$
    for all $\vxs \in \rr^n$,
    so that $\cls{V} \subseteq \cls{\ld(\valpha)}$ by \autoref{prp:suppconju},
    from which
    $V \subseteq \cls{V} \subseteq \ld(\valpha)$ since $\ld(\valpha)$ is closed.
    In turn, to establish the inequality, by \autoref{prp:reachofld} it suffices to show that
    $$
    \dpr{\vxs}{\valpha(\Diaa{\geq})} \geq \dpr{\vxs}{\int_S f \,\di\valpha}
    $$
    for all $\vxs \in \rr^n$ and all measurable $f: S \to [0, 1]$.
    However, this is easily implied by \autoref{prp:xstarDmuintfcomparingzerosigned}
    since $f \geq 0$, $1 - f \geq 0$ on $S$.
\end{proof}

\begin{corollary}\label{col:discretediagram}
    Let $\valpha$ be an $n$-dimensional finite signed measure on $(S, \F)$
    where $\F = \sigma(\A)$ for $\A = \{A_1, \ldots, A_m\}$ a disjoint partition of $S$.
    Then
    $$
    \ld(\valpha) = \Big\{\sum_{i=1}^{m} \lam_i\valpha(A_i) \,:\, \lam_i \in [0, 1], 1 \leq i \leq m\Big\}.
    $$
\end{corollary}

\noindent
Here a \emph{zonoid}, in the range sense,
is the range of a nonatomic real vector measure. Such a range is convex
by the vector-measure range theorem and is therefore a Lorenz hull.
This terminology does not include arbitrary translations that omit zero.
For completeness, we show that the reverse statement also holds.\footnote{A very brief proof
can be found in Bolker \cite{Bolker1969Class} (Theorem 1.6). We hereby present one with details enough to our own satisfaction.}

\begin{proposition}\label{prp:ldiszonoid}
    For any $n$-dimensional finite signed measure $\valpha$ on $(S, \F)$,
    there exists $(S', \F')$ and non-atomic $n$-dimensional finite signed measure
    $\valpha'$ on $(S', \F')$ such that $\lsk(\valpha') = \ld(\valpha)$.
\end{proposition}
\begin{proof}
    Recalling the discussion following \autoref{prp:atomsofvecmeasures},
    let $\{A_i\}_{i \in \nn}$ be a chosen collection of atoms
    with respect to the purely atomic part $\valpha_{|A}$ of $\valpha$,
    and let $\valpha_{|B}$ denote the non-atomic part of $\valpha$.

    Use disjoint tagged copies of the spaces if necessary.
    Consider $A' = (0, \infty)$ and $\B$ the Borel $\sigma$-algebra on $A'$.
    Let $\vbeta$ be the non-atomic finite signed vector measure on $(A', \B)$ such that
    $\vbeta(E) = \mu(E)\valpha(A_i)$ for all Borel subsets $E \subseteq (i-1, i]$,
    where $\mu$ is the Lebesgue measure on Borel subsets of $\rr$, for all $i \in \nn$.
    The defining series has finite total variation because
    $\sum_i\|\valpha(A_i)\|_1\leq|\valpha|(S)<\infty$.
    Then
    \begin{align*}
        \lsk(\vbeta)
        & = \{\vbeta(E) \,:\, E \in \B\} \\
        & = \Big\{\sum_{i=1}^{\infty} \vbeta(E \cap (i - 1, i]) \,:\, E \in \B\Big\} \\
        & = \Big\{\sum_{i=1}^{\infty} \lam_i\valpha(A_i) \,:\, \lam_i \in [0, 1], i \in \nn\Big\}.
    \end{align*}
    On the other hand,
    \begin{align*}
        \ld(\valpha_{|A})
        & = \Big\{\int_A f \,\di\valpha \,:\, f: S \to [0, 1] \text{ is measurable}\Big\} \\
        & = \Big\{\sum_{i=1}^{\infty} \int_{A_i} f \,\di\valpha \,:\, f: S \to [0, 1] \text{ is measurable}\Big\} \\
        & = \Big\{\sum_{i=1}^{\infty} \lam_i\valpha(A_i) \,:\, \lam_i \in [0, 1], i \in \nn\Big\},
    \end{align*}
    where the first equality follows by \autoref{prp:ldavhull},
    where the second equality follows since $|\valpha|(\bigcup_{i=n}^\infty A_i) \to 0$ as $n \to \infty$,
    and where the third equality follows since $f$ is $|\valpha|$-almost constant on each atom $A_i$.
    Thus $\lsk(\vbeta) = \ld(\valpha_{|A})$.

    Let $\valpha' = \vbeta \oplus \valpha_{|B}$. Then $\valpha'$ is non-atomic and
    $$
    \lsk(\valpha') = \lsk(\vbeta) + \lsk(\valpha_{|B}) = \ld(\valpha_{|A}) + \ld(\valpha_{|B}) = \ld(\valpha),
    $$
    where the last equality follow by \autoref{prp:ldsum}.
\end{proof}

\end{spacing}
\end{document}